\documentclass{amsart}
\usepackage{amsmath}
\usepackage{amssymb}
\usepackage{amsthm}
\usepackage{amsfonts}
\usepackage{tikz-cd}
\usepackage{enumitem}
\usepackage[hidelinks]{hyperref}

\newtheorem{theorem}{Theorem}[section]
\newtheorem{proposition}[theorem]{Proposition}
\newtheorem{corollary}[theorem]{Corollary}
\newtheorem{lemma}[theorem]{Lemma}

\theoremstyle{definition}

\newtheorem{definition}[theorem]{Definition}

\theoremstyle{remark}
\newtheorem*{construction}{Construction}
\newtheorem{remark}[theorem]{Remark}

\DeclareMathOperator*{\Spec}{Spec}

\begin{document}
	\title[THE PICARD GROUP OF THE STACK OF SMOOTH POINTED CYCLIC COVERS OF $\mathbb{P}^1$]{THE PICARD GROUP OF THE STACK OF POINTED SMOOTH CYCLIC COVERS OF THE PROJECTIVE LINE}
	\author{ALBERTO LANDI}
	\address{Brown University, 151 Thayer Street, Providence, RI 02912, USA}
	\email{alberto\_landi@brown.edu}
	
	

	
		\begin{abstract}
			We study the stack $\mathcal{H}_{r,g,n}$ of $n$-pointed smooth cyclic covers of degree $r$ between smooth curves of genus $g$ and the projective line. We give two presentations of an open substack of $\mathcal{H}_{r,g,n}$ as a quotient stack, and we study its complement. Using this, we compute the integral Picard group of $\mathcal{H}_{r,g,n}$. Moreover, we obtain a very explicit description of the generators of the Picard group, which have evident geometric meaning. As a corollary of the computation, we get the integral Picard group of the stack $\mathcal{H}_{g,n}$ of $n$-pointed hyperelliptic curves of genus $g$. Finally, taking $g=2$ and recalling that $\mathcal{H}_{2,n}=\mathcal{M}_{2,n}$, we obtain $\mathrm{Pic}(\mathcal{M}_{2,n})$.
		\end{abstract}
	
	\maketitle
		\section{Introduction}
			Let us fix a base field $k$; for simplicity in the introduction we will assume $k=\mathbb{C}$, except when we mention otherwise. The stack of smooth, or stable, $n$-pointed curves of genus $g$ will be denoted by $\mathcal{M}_{g,n}$, respectively $\overline{\mathcal{M}}_{g,n}$.
			
			The problem of computing Picard groups of natural stacks of curves has a long history, starting from Mumford's seminal paper~\cite{Mum65}, where it has been shown that $\mathrm{Pic}(\mathcal{M}_{1,1})=\mathbb{Z}/12\mathbb{Z}$. This was later generalized in~\cite{FO10} to arbitrary base schemes.
			
			Since then a large number of papers have been written on this subject; we will only mention a few of those that are more relevant to our work.
			
			For $g\geq3$ and any $n$, it is shown in~\cite{Mum77} that $\mathrm{Pic}(\mathcal{M}_{g,n})$ and $\mathrm{Pic}(\overline{\mathcal{M}}_{g,n})$ are free abelian groups, while Harer proved that $\mathrm{Pic}(\overline{\mathcal{M}}_{g,n})$ is of rank $n+1$ in~\cite{Har83}; an explicit basis is exhibited in~\cite{AC87}.
			
			In~\cite{Vis98} it was shown by Vistoli that $\mathrm{Pic}(\mathcal{M}_2)=\mathbb{Z}/10\mathbb{Z}$; this was later generalized in~\cite{AV04} to the stack $\mathcal{H}_{g}$ of hyperelliptic curves of genus $g$ (recall $\mathcal{M}_2=\mathcal{H}_2$). There the result is that $\mathrm{Pic}(\mathcal{H}_{g})=\mathbb{Z}/m\mathbb{Z}$, where $m=8g+4$ if $g$ is odd, and $m=4g+2$ when $g$ is even.
			
			Actually, the result in~\cite{AV04} is about the stack $\mathcal{H}_{\text{sm}}(l,r,d)$ of smooth uniform cyclic covers of degree $r$ and branch degree $d$ of projective spaces of dimension $l$. There it is shown that the Picard group of $\mathcal{H}_{\text{sm}}(l,r,d)$ over a field of characteristic not dividing $2rd$ is cyclic of order $r(rd-1)^l\gcd(d,l+1)$. In that article, the main idea is to find a presentation of $\mathcal{H}_{\text{sm}}(l,r,d)$ as a quotient stack of an open subset of a representation of an algebraic group, in order to apply the results about equivariant Chow rings proved in~\cite{EG98}.
			
			We are interested in the case $l=1$, that is the stack of smooth uniform cyclic covers of the projective line, and we extend the results of~\cite{AV04} to the pointed case. In particular, when $r=2$ we get the stack of $n$-pointed hyperelliptic curves, which we simply denote by $\mathcal{H}_{g,n}$. In this direction there are two main articles from which we took inspiration.
			
			Firstly, Scavia (\cite{Sca20}) has computed the rational Picard group of $\mathcal{H}_{g,n}$ over $\mathbb{C}$, showing that it is a vector space of dimension $n$ over $\mathbb{Q}$. This is done using transcendental methods. The result for $n\leq2g+6$ was later generalized to fields of characteristics different from $2$ by Cannings and H. Larson in~\cite[Lemma 5.7]{CL22}, where they computed the rational Chow ring under those assumptions.
			
			Secondly, Pernice (\cite{Per22}) has computed the integral Chow ring of $\mathcal{H}_{g,1}$, showing in particular that $\mathrm{Pic}(\mathcal{H}_{g,1})\simeq\mathbb{Z}\oplus\mathbb{Z}/m\mathbb{Z}$, again with $m=8g+4$ if $g$ is odd, and $m=4g+2$ if $g$ is even.
			
			Clearly, these two articles inspired other works that are relevant for our paper, such as~\cite{EH21} by Edidin and Z. Hu, where the authors computed some relations between natural divisors in $\mathcal{H}_{g,n}$.
			
			We prove that $\mathrm{Pic}(\mathcal{H}_{g,n})\simeq\mathbb{Z}^n\oplus\mathbb{Z}/m\mathbb{Z}$ \/ over any field $k$ of characteristic not dividing $2(g+1)$, where $m$ is the same as above; in particular, we get the integral Picard group of $\mathcal{M}_{2,n}$. This is a corollary of a more general computation of the integral Picard group of $n$-pointed smooth uniform cyclic covers of degree $r$ and branch degree $d$ of the projective line, which we outline now.
			
			Let $k$ be any field and $r,g,n$ non negative integers with $g\geq2$ and $r\geq2$. We denote by $\mathcal{H}_{r,g,n}$ the stack over $k$ where an object over a $k$-scheme $S$ is given by a smooth curve $C$ of genus $g$ over $S$ and a smooth uniform cover $f:C\rightarrow P$ of degree $r$ over a Brauer-Severi scheme $P$ of relative dimension 1 over $S$, together with $n$ mutually disjoint sections of $C\rightarrow S$. See~\cite[Definitions 2.1, 2.3, 2.4]{AV04} for the definitions.
			
			It is known that the stack is non empty if and only if the number $d$ satisfying the relation $r(r-1)d=2g-2+2r$ is in $\mathbb{N}$, see~\cite[Theorem A]{PTT15}. Throughout the whole paper we will always assume that $d\in\mathbb{N}$ and that the characteristic of the base field $k$ does not divide $2rd$. In this case, $\mathcal{H}_{r,g,n}$ is smooth and geometrically connected, and $d$ is the degree of the branch divisor of $f$. 

			The main result of this paper is the following.
			\begin{theorem}\label{thm:mainintro}
				Let $r,g,n$ be non negative integers with $r\geq2$ and $g\geq2$. Suppose moreover that $d=(2g-2+2r)/r(r-1)\in\mathbb{N}$, and that the ground field $k$ is of characteristic not dividing $2rd$.
				\begin{itemize}
					\item If $d$ is even, then \/ $\mathrm{Pic}(\mathcal{H}_{r,g,n})\simeq\mathbb{Z}^{(r-2)\binom{n}{2}+n}\oplus\mathbb{Z}/2r(rd-1)\mathbb{Z}$.
					\item If $d$ is odd, then \/ $\mathrm{Pic}(\mathcal{H}_{r,g,n})\simeq\mathbb{Z}^{(r-2)\binom{n}{2}+n}\oplus\mathbb{Z}/r(rd-1)\mathbb{Z}$.
				\end{itemize}
				If $n<2$, we set $\binom{n}{2}$ to be 0.
			\end{theorem}
			It is worth stating Theorem~\ref{thm:mainintro} in the case of $r=2$, i.e. for the stack $\mathcal{H}_{g,n}$ of $n$-pointed hyperelliptic curves of genus $g$, in which case $d=g+1$. Here the case $n=1$ actually follows from the article~\cite{Per22}.
			\begin{corollary}\label{cor:mainintrohyperelliptic}
				Suppose that the ground field $k$ is of characteristic not dividing $2(g+1)$.
				\begin{itemize}
					\item If $g$ is odd, then \/ $\mathrm{Pic}(\mathcal{H}_{g,n})\simeq\mathbb{Z}^{n}\oplus\mathbb{Z}/(8g+4)\mathbb{Z}$.
					\item If $g$ is even, then \/ $\mathrm{Pic}(\mathcal{H}_{g,n})\simeq\mathbb{Z}^{n}\oplus\mathbb{Z}/(4g+2)\mathbb{Z}$.
				\end{itemize}
			\end{corollary}
			Recall that in the case of $g=2$ one has $\mathcal{H}_{2,n}=\mathcal{M}_{2,n}$, yielding the following Corollary.
			\begin{corollary}\label{cor:maingenus2}
				Suppose that the ground field $k$ is of characteristic not dividing $2$. Then,
				\[	\mathrm{Pic}(\mathcal{M}_{2,n})\simeq\mathbb{Z}^n\oplus\mathbb{Z}/10\mathbb{Z}.
				\]
			\end{corollary}
			This Corollary is also a consequence of~\cite[Theorem B]{FV20}, which computes the Picard group of $\overline{\mathcal{M}}_{g,n}$ and $\mathcal{M}_{g,n}$ for every $g$ and $n$, over fields of characteristic different from 2. Moreover, when $g\leq5$ their result does not require any assumption on the characteristic.
			\subsection*{The general idea.}
			Let
			\[
				\begin{tikzcd}
					\mathcal{C}_{r,g,n}\arrow[r,"f"] & \mathcal{P}_{r,g,n}\arrow[r,"\pi"] & \mathcal{H}_{r,g,n}
				\end{tikzcd}
			\]
			be the universal smooth uniform cyclic cover of degree $r$ and branch divisor of degree $d$ of the universal Brauer-Severi stack of relative dimension 1 over $\mathcal{H}_{r,g,n}$. By the universal Brauer-Severi stack we do not mean the universal family over $\mathcal{M}_{0,n}$, but its base change under the natural map $\mathcal{H}_{r,g,n}\rightarrow\mathcal{M}_{0,n}$. We also have $n$ universal sections $\widetilde{\sigma}_1,\ldots,\widetilde{\sigma}_n$ of $\pi\circ f$, and we can consider the `intermediate' sections $\sigma_1,\ldots,\sigma_n$ given by $\sigma_i=f\circ\widetilde{\sigma}_i$. Then $\sigma_i$ are sections of $\pi$. Although the $\widetilde{\sigma}_i$ have to be disjoint by definition, the $\sigma_i$ are not. This precludes us to always find a presentation of $\mathcal{H}_{r,g,n}$ similar to the ones found in~\cite[Corollary 4.2]{AV04} for $n=0$ and in~\cite[Corollary 1.8]{Per22} for $n=1$, $r=2$.
			
			The idea is to fix a generator $\alpha$ of $\mu_r$, and to consider the closed substacks $\mathcal{Z}^{i,j,l}_{r,g,n}$ of $\mathcal{H}_{r,g,n}$, with $i\not=j$ and $1\leq l\leq r-1$, defined by the locus where the section $\widetilde{\sigma}_j$ is obtained by $\widetilde{\sigma}_i$ applying $\alpha^l$. They are integral and smooth (see Remark~\ref{rmk:Zrgnijlsmooth}), and clearly $\mathcal{Z}_{r,g,n}^{i,j,l}=\mathcal{Z}_{r,g,n}^{j,i,-l}$. We denote with $\mathcal{Z}_{r,g,n}$ their union with the reduced structure. The complementary open substack $\mathcal{H}_{r,g,n}^{\text{far}}$ is exactly the locus over which the intermediate sections $\sigma_i$ are mutually disjoint. Now, by standard methods we get an exact sequence
			\begin{equation}\label{eq:exactintro}
				\begin{tikzcd}
					\bigoplus_{i<j,1\leq l\leq r-1}[\mathcal{Z}^{i,j,l}_{r,g,n}]\mathbb{Z}\arrow[r] & \mathrm{Pic}(\mathcal{H}_{r,g,n})\arrow[r] & \mathrm{Pic}(\mathcal{H}_{r,g,n}^{\text{far}})\arrow[r] & 0
				\end{tikzcd}
			\end{equation}
			hence we reduce to compute $\mathrm{Pic}(\mathcal{H}_{r,g,n}^{\text{far}})$ and all the relations between the $[\mathcal{Z}^{i,j,l}_{r,g,n}]$. The second part is taken care of by Corollary~\ref{cor:relationsZgnij} and Proposition~\ref{prop:injectivityfirstmap}.
			
			Since over $\mathcal{H}_{r,g,n}^{\text{far}}$ the intermediate sections are disjoint, we can extend the ideas of Pernice (\cite{Per22}) to find a presentation of $\mathcal{H}_{r,g,n}^{\text{far}}$ as a quotient stack of a locally closed subscheme of a product of affine spaces by an action of an affine smooth algebraic group $G_{r,g,n}$, see corollaries~\ref{cor:quotientclosedHgnfar1},~\ref{cor:quotientclosedHgnfar2} and~\ref{cor:quotientclosedHgnfar}. For $n=2,3$, this has already been done in~\cite{EH22}.
			
			For $n\leq rd+1$ we can find another presentation of $\mathcal{H}_{r,g,n}^{\text{far}}$, this time as a quotient stack of an open subscheme of a representation of $G_{r,g,n}$ (see Corollary~\ref{prop:quotientopenHgnfar}), similarly to what was done in~\cite{Per22}. This enables us to compute $\mathrm{Pic}(\mathcal{H}_{r,g,n}^{\text{far}})$ using the techniques of~\cite{AV04}, which are based on the theory of equivariant Chow groups developed in~\cite{EG98}.
			
			We cannot proceed in the same fashion for $n>rd+1$. In this case, we exploit the fact that $\mathcal{H}_{r,g,n}^{\text{far}}$ is a scheme for $n\geq rd+1$. Indeed, this allow us to compute its Picard group using the Grothendieck-Lefschetz Theorem, see Proposition~\ref{prop:surjectivepicard}.
			
			Finally, we reconstruct the Picard group of $\mathcal{H}_{r,g,n}$ from the computations above. To do that, in the case of $d$ even we find a section of $\mathrm{Pic}(\mathcal{H}_{r,g,n})\rightarrow\mathrm{Pic}(\mathcal{H}_{r,g,n}^{\text{far}})$ (see Lemma~\ref{lem:comparisonpicardHg2farHg1} and Lemma~\ref{lem:strongcomparisonpicardHgnfarHg0}). However, when $d$ is odd such section does not exist, and a bit more work is needed (see Lemma~\ref{lem:torsiondodd}). Putting everything together we get Theorem~\ref{thm:mainintro}.
			
			\subsection*{The structure of the paper.}
			In Section~\ref{sec:geometricdescription} we study the closed substacks $\mathcal{Z}_{r,g,n}^{i,j,l}$, in particular we find some relations between the divisors associated to them. Then we give the two presentations of $\mathcal{H}_{r,g,n}^{\text{far}}$ as a quotient stack, the second working only for $n\leq rd+1$. Moreover, we study an important divisor $\Delta_{r,g,n}$ of a scheme related to $\mathcal{H}_{r,g,n}^{\text{far}}$ for $n\leq rd+1$, which is relevant for the computation of the Picard group of $\mathcal{H}_{r,g,n}^{\text{far}}$.
			
			In Sections~\ref{sec:case<=rd+1} and~\ref{sec:>=2g+3} we compute $\mathrm{Pic}(\mathcal{H}_{r,g,n}^{\text{far}})$ for $1\leq n\leq rd+1$ and $n\geq rd+1$ respectively, using the presentations found in Section 2.
			
			In Section~\ref{sec:finalcomputation} we show that all the relations between the divisors $\mathcal{Z}_{r,g,n}^{i,j,l}$ are exactly those found in Section 2.
			Moreover, when $d$ is even we find a section of $\mathrm{Pic}(\mathcal{H}_{r,g,n})\rightarrow\mathrm{Pic}(\mathcal{H}_{r,g,n}^{\text{far}})$, concluding the computation in that case. Then, we deal with the case $d$ odd, which requires additional work. Finally, in the last subsection we describe a set of generators of the Picard group that have geometric meaning.
		\subsection*{Acknowledgements}
			I would like to thank my advisor Angelo Vistoli for suggesting the problem and for his guidance throughout the work. I would also like to thank Dan Edidin and Mattia Talpo for very useful discussions.
		\section{Description of $\mathcal{Z}_{r,g,n}^{i,j,l}$ and $\mathcal{H}_{r,g,n}^{\mathrm{far}}$}\label{sec:geometricdescription}
			We denote by $\mathcal{H}_{r,g,n}$ the stack where an object over a $k$-scheme $S$ is given by a $n$-pointed smooth curve $C$ of genus $g$ over $S$ and a smooth uniform cover $f:C\rightarrow P$ of degree $r$ of a Brauer-Severi scheme $P$ of relative dimension 1 over $S$, together with $n$ mutually disjoint sections of $C\rightarrow S$. See~\cite[Definitions 2.1, 2.3, 2.4]{AV04} for the definitions. Notice that the stack is non empty if and only if the number $d$ satisfying $r(r-1)d=2g-2+2r$ is in $\mathbb{N}$, see~\cite[Theorem A]{PTT15}. We will always assume that $d\in\mathbb{N}$ and that the characteristic of the ground field $k$ does not divide $2rd$, in which case $d$ is the degree of the branch divisor of $f$. Notice that for $r=2$ we have that $\mathcal{H}_{2,g,n}=:\mathcal{H}_{g,n}$ is the stack of $n$-pointed hyperelliptic curves, which is the case of most interest. For $n=0$, the stack $\mathcal{H}_{r,g,0}=:\mathcal{H}_{r,g}$ has been studied in the article~\cite{AV04}, where it was given a presentation of that stack as a quotient of an open subscheme of a representation of $\mathrm{GL}_2/\mu_{d}$, proving in particular that it is smooth and geometrically connected. Since the forgetful map $\mathcal{H}_{r,g,n}\rightarrow\mathcal{H}_{r,g}$ is smooth with geometrically connected fibers, $\mathcal{H}_{r,g,n}$ is smooth and geometrically connected for every $n\geq0$.
			
			Now, let
			\[
			\begin{tikzcd}
				\mathcal{C}_{r,g,n}\arrow[r,"f"] & \mathcal{P}_{r,g,n}\arrow[r,"\pi"] & \mathcal{H}_{r,g,n}
			\end{tikzcd}
			\]
			be the universal smooth uniform cyclic cover of degree $r$ and branch divisor of degree $d$ over the universal Brauer-Severi scheme of relative dimension 1 over $\mathcal{H}_{r,g,n}$.
			We also have $n$ universal sections $\widetilde{\sigma}_1,\ldots,\widetilde{\sigma}_n$ of $\pi\circ f$; consider the "intermediate" sections $\sigma_1,\ldots,\sigma_n$ given by $\sigma_i=f\circ\widetilde{\sigma}_i$. Then $\sigma_i$ are sections of $\pi$.
			We want to define rigorously the substacks $\mathcal{Z}_{r,g,n}^{i,j,l}$ and $\mathcal{H}_{r,g,n}^{\text{far}}$ mentioned in the introduction.
			\begin{construction}
			First of all, notice that $\sigma_i(\mathcal{H}_{r,g,n})$ defines a relative effective Cartier divisor on $\mathcal{P}_{r,g,n}$ over $\mathcal{H}_{r,g,n}$. Define the divisor $\mathcal{D}_{r,g,n}$ to be the sum of the $\sigma_i(\mathcal{H}_{r,g,n})$, which is therefore a relative effective Cartier divisor.
			It follows that the projection
			\[
				\begin{tikzcd}
					\pi|_{\mathcal{D}_{r,g,n}}:\mathcal{D}_{r,g,n}\arrow[r] & \mathcal{H}_{r,g,n}
				\end{tikzcd}
			\]
			is flat and finite of degree $n$. Let $\mathcal{H}_{r,g,n}^{\text{far}}\subset\mathcal{H}_{r,g,n}$ be the open substack over which that map is étale. Clearly, it is the locus over which the sections $\sigma_i$ are mutually disjoint.
			
			Let $\mathcal{Z}_{r,g,n}$ be the complement, with the following stack structure. Let $\mathcal{C}_{r,g}\rightarrow\mathcal{H}_{r,g}$ be the universal curve over $\mathcal{H}_{r,g}$. Then, $\mathcal{H}_{r,g,n}$ can be seen as the complement of the extended diagonal $\Delta=\bigcup\Delta_{i,j}$ in the $n$-fold fibered product
			\[
				C_{r,g}^n:=C_{r,g}\times_{\mathcal{H}_{r,g}}\ldots\times_{\mathcal{H}_{r,g}}C_{r,g}.
			\]
			Let $\alpha$ be a fixed generator of $\mu_{r}$, and $\alpha_i$ the automorphism of $C_{r,g}^n$ given by $\alpha$ on the $i$-th component and the identity on the others. Then, for all $i\not=j$ and $1\leq l\leq r-1$ we define
			\[
				\mathcal{Z}_{r,g,n}^{i,j,l}:=(\alpha_i^l)^{-1}(\Delta_{i,j})\cap\mathcal{H}_{r,g,n}.
			\]
			In particular $\mathcal{Z}_{r,g,n}^{i,j,l}$ is integral and smooth, since $\Delta_{i,j}\simeq\mathcal{C}_{r,g}^{n-1}$ which is integral and smooth, and obviously $\mathcal{Z}_{r,g,n}^{i,j,l}=\mathcal{Z}_{r,g,n}^{j,i,-l}$, so we usually consider $i<j$. Clearly, $\mathcal{Z}_{r,g,n}^{i,j}:=\bigcup_{l=1}^{r-1}\mathcal{Z}_{r,g,n}^{i,j,l}$ is the locus over which $\sigma_i$ and $\sigma_j$ intersects, and $\mathcal{Z}_{r,g,n}$ is the union of these substacks as $i\not=j$ vary. The natural choice is then to give $\mathcal{Z}_{r,g,n}$ and $\mathcal{Z}_{r,g,n}^{i,j}$ the reduced structure.
		\end{construction}		
			In the following subsection we study the divisors $\mathcal{Z}_{r,g,n}^{i,j,l}$, and we compute some relations between their classes in $\mathrm{Pic}(\mathcal{H}_{r,g,n})$. These are actually the only relations, as we will see in Section~\ref{sec:finalcomputation}, Proposition~\ref{prop:injectivityfirstmap}. In the second and third subsections we give two different presentations of $\mathcal{H}_{r,g,n}^{\text{far}}$ as a quotient stack, the second working only for $n\leq rd+1$. Both presentations will be used to compute the Picard group of that stack, see Sections~\ref{sec:case12},~\ref{sec:3<=n<=2g+3} and~\ref{sec:>=2g+3}. In the last subsection we study an important divisor $\Delta_{r,g,n}$ of a scheme related to $\mathcal{H}_{r,g,n}^{\text{far}}$ for $n\leq rd+1$, which is relevant for the computation of the Picard group of $\mathcal{H}_{r,g,n}^{\text{far}}$.
			\subsection{Description of $\mathcal{Z}_{r,g,n}^{i,j,l}$}
			We start by making some remarks about the geometry of $\mathcal{Z}_{r,g,n}^{i,j,l}$ and $\mathcal{Z}_{r,g,n}^{i,j}$.
			\begin{remark}\label{rmk:Zrgnijlsmooth}
				By construction, we have seen that for all $i\not=j$ the closed substack $\mathcal{Z}_{r,g,n}^{i,j,l}$ is integral and smooth. Moreover, it is possible to see $\mathcal{Z}_{r,g,n}^{i,j,l}$ as an open substack of $\mathcal{H}_{r,g,n-1}$. Indeed, we have isomorphisms $(\alpha_i^l)^{-1}(\Delta_{i,j})\simeq\Delta_{i,j}\simeq\mathcal{C}_{r,g}^{n-1}$, which induce an isomorphism between $\mathcal{Z}_{r,g,n}^{i,j,l}=(\alpha_i^l)^{-1}(\Delta_{i,j})\cap\mathcal{H}_{r,g,n}$ and an open substack of $\mathcal{H}_{r,g,n-1}$. Finally, the objects of $\mathcal{Z}_{r,g,n}^{i,j,l}$ over a $k$-scheme $S$ are objects of $\mathcal{H}_{r,g,n}$ where the section $\widetilde{\sigma}_j$ is obtained by $\widetilde{\sigma}_i$ applying $\alpha^l$.
			\end{remark}
			\begin{remark}
				Notice that two different $\mathcal{Z}_{r,g,n}^{i,j,l}$ and $\mathcal{Z}_{r,g,n}^{i',j',l'}$ do not always intersect. For instance, if $i=i'$, $j=j'$ and $l\not=l'$, then they are disjoint, hence varying $l$ the $\mathcal{Z}_{r,g,n}^{i,j,l}$ are actually the connected components of $\mathcal{Z}_{r,g,n}^{i,j}$, which is thus smooth. This shows that the objects of $\mathcal{Z}_{r,g,n}^{i,j}$ over a $k$-scheme $S$ are objects of $\mathcal{H}_{r,g,n}$ where $\sigma_i=\sigma_j$. Finally, at most $r$ sections of $\pi$ can all intersect at one point. In particular, in the case of $n=3$ and $r=2$, the closed substacks $\mathcal{Z}_{g,3}^{1,2}$, $\mathcal{Z}_{g,3}^{1,3}$, $\mathcal{Z}_{g,3}^{2,3}$ are actually the connected components of $\mathcal{Z}_{g,3}$.
			\end{remark}	
			Now, Remark~\ref{rmk:Zrgnijlsmooth} has the following consequence. In general, since $\mathcal{H}_{r,g,n}$ is smooth, we have an exact sequence
			\[
			\begin{tikzcd}
				K\arrow[r] & \mathrm{Pic}(\mathcal{H}_{r,g,n})\arrow[r] & \mathrm{Pic}(\mathcal{H}_{r,g,n}^{\text{far}})\arrow[r] & 0
			\end{tikzcd}				
			\]
			where the kernel $K$ is the free group generated by the classes of irreducible components of $\mathcal{Z}_{r,g,n}$, see~\cite[Proposition 1.9]{PTT15}. Since the $\mathcal{Z}_{r,g,n}^{i,j,l}$ are integral, the above exact sequence is
			\begin{equation}\label{eq:exactcoarse}
				\begin{tikzcd}
					\bigoplus_{i<j,1\leq l\leq r-1}[\mathcal{Z}^{i,j,l}_{r,g,n}]\mathbb{Z}\arrow[r] & \mathrm{Pic}(\mathcal{H}_{r,g,n})\arrow[r] & \mathrm{Pic}(\mathcal{H}_{r,g,n}^{\text{far}})\arrow[r] & 0.
				\end{tikzcd}
			\end{equation}
			Consider for a moment the case $r=2$, where necessarily $l=1$. Motivated by the fact that the rational Picard group of $\mathcal{H}_{g,n}$ over $\mathbb{C}$ is a vector space of dimension $n$, see~\cite[Theorem 1.1]{Sca20}, we would like to find relations between the classes $[\mathcal{Z}^{i,j}_{g,n}]$, in order to make the number of generators of the image drop to less or equal to $n$. Actually, we will later prove more generally that the rational Picard group of $\mathcal{H}_{r,g,n}$ has rank $(r-2)\binom{n}{2}+n$ over all fields of characteristic not dividing $2rd$, see Section~\ref{sec:finalcomputation}, hence providing an extension of the already mentioned~\cite[Theorem 1.1]{Sca20} (see also~\cite[Lemma 5.7]{CL22}). Of course, we find the relations for all $r\geq2$.
			
			First, we compute the invertible ideal sheaf of $\mathcal{Z}^{i,j}_{r,g,n}$, for all $i\not=j$. Define $\mathcal{D}_{r,g,n}^{i,j}$ as the relative effective Cartier divisor given by $\sigma_i(\mathcal{H}_{r,g,n})+\sigma_j(\mathcal{H}_{r,g,n})$, and consider the projection
			\[
				\begin{tikzcd}
					\pi|_{\mathcal{D}_{r,g,n}^{i,j}}:\mathcal{D}_{r,g,n}^{i,j}\arrow[r] & \mathcal{H}_{r,g,n}.
				\end{tikzcd}
			\]
			Then $\mathcal{Z}_{r,g,n}^{i,j}$ is the ramification locus of this map, with the reduced structure. There is another natural choice for the structure of $\mathcal{Z}_{r,g,n}^{i,j}$. Indeed, thanks to the description of flat double covers in characteristic different from 2 (see~\cite{AV04} for the more general case of uniform cyclic covers), the ramification locus can be also defined as the zero locus of a section of $(\det\mathcal{A}_{r,g,n}^{i,j})^{-2}$, called the discriminant section, where $\mathcal{A}_{r,g,n}^{i,j}=(\pi|_{\mathcal{D}_{r,g,n}^{i,j}})_*\mathcal{O}_{\mathcal{D}_{r,g,n}^{i,j}}$. The discriminant section is stable under base change. Of course, the structure as a stack that we obtain might be different from the one of $\mathcal{Z}_{r,g,n}^{i,j}$. Indeed, we prove that this is the case, by showing that the ideal sheaf of $\mathcal{Z}_{r,g,n}^{i,j}$ is $\det\mathcal{A}_{r,g,n}^{i,j}$, instead of $(\det\mathcal{A}_{r,g,n}^{i,j})^2$. First we study the intersection of the image of two sections of $\pi:\mathcal{P}_{r,g,n}\rightarrow\mathcal{H}_{r,g,n}$.
			\begin{lemma}\label{lem:intersectionofsections}
				Let $\mathcal{R}_{r,g,n}^{i,j}\subset\mathcal{P}_{r,g,n}$ the schematic intersection of $\sigma_i(\mathcal{H}_{r,g,n})$ and $\sigma_j(\mathcal{H}_{r,g,n})$. Then, $\mathcal{R}_{r,g,n}^{i,j}$ is isomorphic to $\mathcal{Z}_{r,g,n}^{i,j}$ via $\pi$, with inverse $\sigma_i$. In particular $\mathcal{R}_{r,g,n}^{i,j}$ is smooth.
			\end{lemma}
			\begin{proof}
				It is enough to show that the diagram
				\[
				\begin{tikzcd}
					\mathcal{Z}_{r,g,n}^{i,j}\arrow[r,hookrightarrow]\arrow[d,hookrightarrow] & \mathcal{H}_{r,g,n}\arrow[d,"\sigma_j"]\\
					\mathcal{H}_{r,g,n}\arrow[r,"\sigma_i"'] & \mathcal{P}_{r,g,n}
				\end{tikzcd}
				\]
				is cartesian. This is essentially the definition of $\mathcal{Z}_{r,g,n}^{i,j}$.
			\end{proof}
			\begin{proposition}\label{prop:idealsheafZgnij}
				The ideal sheaf of $\mathcal{Z}_{r,g,n}^{i,j}$ is $\det\mathcal{A}_{r,g,n}^{i,j}$.
			\end{proposition}
			\begin{proof}
				The following argument could be made more general. Set
				\[
					\mathcal{D}_{r,g,n}^i:=\sigma_i(\mathcal{H}_{r,g,n})
				\]
				and similarly for $j$. We have a short exact sequence
				\[
					\begin{tikzcd}[row sep=small, column sep=normal]
						0\arrow[r] & \mathcal{O}_{\mathcal{D}_{r,g,n}^{i,j}}\arrow[r] & \mathcal{O}_{\mathcal{D}_{r,g,n}^i}\oplus\mathcal{O}_{\mathcal{D}_{r,g,n}^j}\arrow[r] & \mathcal{O}_{\mathcal{R}_{r,g,n}^{i,j}}\arrow[r] & 0\\
						& & (f,f')\arrow[u,phantom,sloped,"\in"]\arrow[r,mapsto] & (f|_{\mathcal{R}_{r,g,n}^{i,j}}-f'|_{\mathcal{R}_{r,g,n}^{i,j}})\arrow[u,phantom,sloped,"\in"]
					\end{tikzcd}
				\]
				Applying $\pi_*$ we get
				\[
				\begin{tikzcd}[row sep=small, column sep=normal]
					0\arrow[r] & \mathcal{A}_{r,g,n}^{i,j}\arrow[r] & \mathcal{O}_{\mathcal{H}_{r,g,n}}\oplus\mathcal{O}_{\mathcal{H}_{r,g,n}}\arrow[r] & \mathcal{O}_{\mathcal{Z}_{r,g,n}^{i,j}}\arrow[r] & 0.
				\end{tikzcd}
				\]
				Let $I_{\mathcal{Z}_{r,g,n}^{i,j}}$ be the ideal sheaf of $\mathcal{Z}_{r,g,n}^{i,j}$ in $\mathcal{H}_{r,g,n}^{i,j}$. The second map of the exact sequence factors through the automorphism of $\mathcal{O}_{\mathcal{H}_{r,g,n}}\oplus\mathcal{O}_{\mathcal{H}_{r,g,n}}$ given by $(f,f')\mapsto(f,f-f')$, followed by the second projection. Therefore, we have that $\mathcal{A}_{r,g,n}^{i,j}\simeq\mathcal{O}_{\mathcal{H}_{r,g,n}}\oplus I_{\mathcal{Z}_{r,g,n}^{i,j}}$, hence $\det\mathcal{A}_{r,g,n}^{i,j}\simeq I_{\mathcal{Z}_{r,g,n}^{i,j}}$, as wanted.
			\end{proof}
			Thanks to Corollary~\ref{prop:idealsheafZgnij}, we have reduced the problem of finding relations between the classes $[\mathcal{Z}_{r,g,n}^{i,j}]:=\sum_{l}[\mathcal{Z}_{r,g,n}^{i,j,l}]$ to finding relations between the sheaves $\det\mathcal{A}_{r,g,n}^{i,j}$.
			\begin{lemma}\label{lem:descriptionAgnij}
				Let $L_i$ be the dual of the invertible ideal sheaf of $\sigma_i(\mathcal{H}^{i,j}_{r,g,n})$ in $\mathcal{P}_{r,g,n}$, and let $\omega_{\pi}$ be the dualizing sheaf of $\pi$. Then,
				\[
					\det\mathcal{A}_{r,g,n}^{i,j}\simeq\pi_*(L_i\otimes L_j\otimes\omega_{\pi})^\vee.
				\]
			\end{lemma}
			\begin{proof}
				Consider the exact sequence
				\[
				\begin{tikzcd}
					0\arrow[r] & L_i^\vee\otimes L_j^\vee\arrow[r] & \mathcal{O}_{\mathcal{P}_{r,g,n}}\arrow[r] & \mathcal{O}_{D_{r,g,n}^{i,j}}\arrow[r] & 0.
				\end{tikzcd}
				\]
				Take the long exact sequence associated to $\pi_*$ and apply~\cite[Lemma 1.10]{PTT15}. We get that $\pi_*(L_i^\vee\otimes L_j^\vee)=\mathrm{R}^1\pi_*(\mathcal{O}_{\mathcal{P}_{r,g,n}})=0$, and we have an exact sequence
				\[
				\begin{tikzcd}
					0\arrow[r] & \mathcal{O}_{\mathcal{H}_{r,g,n}}\arrow[r] & \mathcal{A}_{r,g,n}^{i,j}\arrow[r] & \mathrm{R}^1\pi_*(L_i^\vee\otimes L_j^\vee)\arrow[r] & 0.
				\end{tikzcd}
				\]
				Moreover, again by~\cite[Lemma 1.10]{PTT15}, we have
				\[
					\mathrm{R}^1\pi_*(L_i^\vee\otimes L_j^\vee)\simeq\pi_*(L_i\otimes L_j\otimes\omega_{\pi})^\vee
				\]
				and it is an invertible sheaf. Applying the determinant yields the desired isomorphism.
			\end{proof}
			Clearly, for all $i\not=j$ the invertible sheaf $L_i\otimes L_j\otimes\omega_\pi$ is isomorphic to
			\[
			(L_1\otimes L_i\otimes\omega_\pi)\otimes(L_1\otimes L_j\otimes\omega_\pi)\otimes(L_2\otimes L_n\otimes\omega_\pi)\otimes(L_1\otimes L_2\otimes\omega_\pi)^\vee\otimes(L_1\otimes L_n\otimes\omega_\pi)^\vee
			\]
			If these relations passed through $\pi_*$, we would have found relations between the $\det\mathcal{A}_{r,g,n}^{i,j}$. Thankfully, this is the case, as the following Lemma shows.
			\begin{lemma}\label{lem:pi*deg0}
				Let $\mathcal{S}$ be an algebraic stack, $\pi:\mathcal{P}\rightarrow\mathcal{S}$ a Brauer-Severi stack of relative dimension 1 over $\mathcal{S}$, and $F_1$, $F_2$ two invertible sheaves of degree 0 on each fiber. Then $\pi_*F_1$, $\pi_*F_2$ and $\pi_*(F_1\otimes F_2)$ are invertible sheaves, their formation commutes with base change, and the natural map
				\[
				\begin{tikzcd}
					\pi_*F_1\otimes \pi_*F_2\arrow[r] & \pi_*(F_1\otimes F_2)
				\end{tikzcd}
				\]
				is an isomorphism.
			\end{lemma}
			\begin{proof}
				The first two statements follow from~\cite[Lemma 1.10]{PTT15}. Moreover, we know that $\pi_*\mathcal{O}_{\mathcal{P}}\simeq\mathcal{O}_{\mathcal{S}}$, 
				therefore to conclude it is enough to show that
				\[
				\begin{tikzcd}
					\pi^*\pi_*F_1\otimes\pi^*\pi_*F_2\arrow[r] & \pi^*\pi_*(F_1\otimes F_2)
				\end{tikzcd}
				\]
				is an isomorphism. Since this is a map between invertible sheaves which satisfy base change, it suffice to show that the map is surjective on the geometric fibers.
				On the geometric fibers the sheaves $F_1$, $F_2$ and $F_1\otimes F_2$ are isomorphic to the structure sheaf of the fiber, which is a projective line. It follows that on the geometric fiber over a point $s$ the previous map is
				\[
				\begin{tikzcd}
					\pi_s^*{\pi_s}_*\mathcal{O}_{\mathbb{P}^1}\otimes\pi_s^*{\pi_s}_*\mathcal{O}_{\mathbb{P}^1}\arrow[r] & \pi_s^*{\pi_s}_*(\mathcal{O}_{\mathbb{P}^1}\otimes\mathcal{O}_{\mathbb{P}^1})
				\end{tikzcd}
				\]
				which is an isomorphism.
			\end{proof}
			\begin{corollary}\label{cor:relationsZgnij}
				Let $i\not=j$ and both different from 1, and $n>2$. Then the following holds
				\[
				\det\mathcal{A}^{i,j}_{r,g,n}\simeq\det\mathcal{A}^{1,i}_{r,g,n}\otimes\det\mathcal{A}^{1,j}_{r,g,n}\otimes\det\mathcal{A}^{2,n}_{r,g,n}\otimes(\det\mathcal{A}^{1,2}_{r,g,n})^\vee\otimes(\det\mathcal{A}^{1,n}_{r,g,n})^\vee
				\]
				that is
				\[
				[\mathcal{Z}^{i,j}_{r,g,n}]=[\mathcal{Z}^{1,i}_{r,g,n}]+[\mathcal{Z}^{1,j}_{r,g,n}]+[\mathcal{Z}^{2,n}_{r,g,n}]-[\mathcal{Z}^{1,2}_{r,g,n}]-[\mathcal{Z}^{1,n}_{r,g,n}]
				\]
				as class of divisors. In particular, the kernel of the restriction homomorphism $\mathrm{Pic}(\mathcal{H}_{r,g,n})\rightarrow\mathrm{Pic}(\mathcal{H}_{r,g,n}^{\text{far}})$ is generated by $(r-2)\binom{n}{2}+n$ class of divisors.
			\end{corollary}
			\begin{proof}
				The relations between the sheaves $L_i\otimes L_j\otimes\omega_{\pi}$ pass through $\pi_*$ thanks to Lemma~\ref{lem:pi*deg0}.
			\end{proof}
			Of course one has to choose the divisors in such a way that they generate. We will see in Section~\ref{sec:finalcomputation} that these are the only relations between the classes $[\mathcal{Z}_{r,g,n}^{i,j,l}]$.
			\begin{corollary}
				If $n=2$, then there is an exact sequence
				\begin{equation}\label{eq:exactmain2}
					\begin{tikzcd}
						\mathbb{Z}^{r-1}\arrow[r] & \mathrm{Pic}(\mathcal{H}_{r,g,2})\arrow[r] & \mathrm{Pic}(\mathcal{H}_{r,g,2}^{\text{far}})\arrow[r] & 0.
					\end{tikzcd}
				\end{equation}
				If $n\geq3$, then there exists an exact sequence
				\begin{equation}\label{eq:exactmain}
					\begin{tikzcd}
						\mathbb{Z}^{(r-2)\binom{n}{2}+n}\arrow[r] & \mathrm{Pic}(\mathcal{H}_{r,g,n})\arrow[r] & \mathrm{Pic}(\mathcal{H}_{r,g,n}^{\text{far}})\arrow[r] & 0.
					\end{tikzcd}
				\end{equation}
			\end{corollary}
			\begin{proof}
				It follows immediately from Corollary~\ref{cor:relationsZgnij} applied to the exact sequence~\eqref{eq:exactcoarse}.
			\end{proof}
			\begin{remark}
				Suppose $r=2$ and $n\geq3$. Then the exact sequence~\eqref{eq:exactmain} becomes
				\begin{equation}\label{eq:exactmainhyperelliptic}
					\begin{tikzcd}
						\left({\bigoplus_{i=2}^{n}[\mathcal{Z}_{g,n}^{1,i}]}\mathbb{Z}\right)\oplus[\mathcal{Z}_{g,n}^{2,n}]\mathbb{Z}\arrow[r] & \mathrm{Pic}(\mathcal{H}_{g,n})\arrow[r] & \mathrm{Pic}(\mathcal{H}_{g,n}^{\text{far}})\arrow[r] & 0.
					\end{tikzcd}
				\end{equation}
				Moreover, given a set of $n$ generators, it is fairly easy to write every class of divisors in terms of these generators, using the relations of Corollary~\ref{cor:relationsZgnij}. Indeed, in place of $\mathcal{Z}_{g,n}^{i,j}$ write simply the pair of indices $(i,j)$. Then every linear combination of these pairs is a relation generated by the ones in Corollary~\ref{cor:relationsZgnij} if and only if every index appears with total multiplicity 0. It is also possible to make sense of pairs with equal indices, defining $(j,j)=(j,u)+(j,v)-(u,v)$ for $j\not=u\not=v\not=j$ (it does not depend on the choice of $u$ and $v$), and the above characterization still holds. Using this, it is immediate to see that for every $j$ the set $[\mathcal{Z}_{g,n}^{j,1}],\ldots,[\mathcal{Z}_{g,n}^{j,n}]$ (repeating $j$ too) generates.
			\end{remark}
			\subsection{Description of $\mathcal{H}_{r,g,n}^{\mathrm{far}}$ and first presentation}\label{subsec:firstpresentation}
					Recall $\mathcal{H}_{r,g,n}^{\text{far}}$ is the open substack of $\mathcal{H}_{r,g,n}$ over which the sections $\sigma_i$ are mutually disjoint. The purpose of this subsection is to find a nice presentation of $\mathcal{H}_{r,g,n}^{\text{far}}$ as a quotient stack of a scheme by the action of an affine smooth algebraic group $G_{r,g,n}$. In the next subsection, for $1\leq n\leq rd+1$, we will manage to describe $\mathcal{H}_{r,g,n}^{\text{far}}$ as the quotient stack of an open subscheme of a representation of $G_{r,g,n}$.
					
					We follow and extend the ideas in~\cite{AV04} and~\cite{Per22}. We know that $\mathcal{H}_{r,g}$ is equivalent to the fibered category $\mathcal{H}'_{r,g}$ whose objects over a $k$-scheme $S$ are tuples
					\[
						(\pi:P\rightarrow S,L,i:L^{\otimes r}\hookrightarrow\mathcal{O}_P)
					\]
					where $P\rightarrow S$ is a Brauer-Severi scheme of relative dimension 1, $L$ is an invertible sheaf of $P$ which restricts to an invertible sheaf of degree $-d$ on any geometric fiber, and $i$ is an injective morphism which remains injective when restricted to any geometric fiber and such that the effective Cartier divisor $\Delta_i$ associated to $i$ is smooth over $S$. For details, see~\cite[Proposition 3.4]{AV04}. The smooth uniform cyclic cover $C\rightarrow P$ corresponding to such an object can be recovered as
					\[
						\begin{tikzcd}
							f:C=\underline{\Spec}_{\mathcal{O}_P}(\mathcal{O}_P\oplus L\oplus\cdots\oplus L^{\otimes r-1})\arrow[r] & P.
						\end{tikzcd}
					\]
					Here the structure of $\mathcal{O}_P\oplus L\oplus\cdots\oplus L^{\otimes r-1}$ as an $\mathcal{O}_P$-algebra is the one induced by $i$.
					Let $\widetilde{\sigma}:S\rightarrow C$ be a section of $\pi\circ f$ and consider the section $\sigma=f\circ\widetilde{\sigma}:S\rightarrow P$ to $\pi$. Then $\widetilde{\sigma}$ corresponds to a morphism of $\mathcal{O}_S$-algebras $\widetilde{\jmath}:\mathcal{O}_S\oplus\sigma^*L\oplus\cdots\oplus \sigma^*L^{\otimes r-1}\rightarrow\mathcal{O}_S$
					thanks to the functorial bijective map
					\[
						\begin{tikzcd}[column sep=16pt]
							\mathrm{Hom}_P(S,\underline{\Spec}_{\mathcal{O}_P}(\mathcal{O}_P\oplus\cdots\oplus L^{\otimes r-1}))\arrow[r] & \mathrm{Hom}_{\mathcal{O}_S}(\sigma^*(\mathcal{O}_P\oplus\cdots\oplus L^{\otimes r-1}),\mathcal{O}_S).
						\end{tikzcd}
					\]
					Now, $\widetilde{\jmath}$ corresponds to a morphism $j:L\rightarrow\mathcal{O}_S$ such that $j^r=\sigma^*(i)$ (since $\widetilde{\jmath}$ is a morphism of $\mathcal{O}_S$-algebras). As always, we refer to~\cite{Per22} for more details in the case $r=2$.
					
					Define $\mathcal{H}'_{r,g,n}$ to be the fibered category whose objects are
					\[
						(\pi:P\rightarrow S,L,i:L^{\otimes r}\hookrightarrow\mathcal{O}_P,\sigma_1,\ldots,\sigma_n,j_1,\ldots,j_n)
					\]
					with the obvious notation, where we ask the sections encoded by $(\sigma_1,j_1),\ldots,$$(\sigma_n,j_n)$ to be mutually disjoint. The arrows are defined in the natural way.
					The proof of the following proposition is a straightforward extension of~\cite[Proposition 1.3]{Per22}, which addresses the case $n=1$, $r=2$.
					\begin{proposition}\label{prop:equivalence}
						There is an equivalence of fibered categories between $\mathcal{H}_{r,g,n}$ and $\mathcal{H}_{r,g,n}'$.\qed
					\end{proposition}
					
					Now, to find the desired presentation, we proceed similarly to what was done in~\cite{Per22}. The idea is to define a torsor $\widetilde{\mathcal{H}}_{r,g,n}^{\text{far}}\rightarrow\mathcal{H}_{r,g,n}^{\text{far}}$ from a scheme, obtained by considering isomorphisms between the Brauer-Severi scheme with $L$ and its first sections, and the projective line with $\mathcal{O}(-d)$ and the $\infty$, 0 and $[1:1]$ sections. Of course, in the case $n=2$ we consider just the sections $\infty$ and 0, and if $n=1$ then just the first one. In order for this to work, it is crucial to consider only $\mathcal{H}_{r,g,n}^{\text{far}}$ and not the whole $\mathcal{H}_{r,g,n}$, since in general two sections $\sigma_i$ and $\sigma_j$ may intersect each other, while for instance the $\infty$ and $0$ sections do not. We make all of this precise.
					\begin{definition}
						Let $\mathbb{A}(m)$ be the affine space of homogeneous polynomials of degree $m$ in two variables, which has dimension $m+1$, and let $\mathbb{A}_{\text{sm}}(m)$ be the open affine subscheme of $\mathbb{A}(m)$ given by the complement of the discriminant locus. Finally, let $\mathbb{A}_{r,n,\text{sm}}(m)$ to be the closed subscheme of $\mathbb{A}_{\text{sm}}(m)\times(\mathbb{A}^{n-3}\setminus\widetilde{\Delta})\times\mathbb{A}^n$ defined by the equations
						\[
						f(0,1)=s_1^r,\ f(1,0)=s_2^r,\ f(1,1)=s_3^r,\ f(1,p_1)=t_1^r,\ \ldots,\ f(1,p_{n-3})=t_{n-3}^r
						\]
						where $f\in\mathbb{A}_{\text{sm}}(m)$, $(p_1,\ldots,p_{n-3})\in\mathbb{A}^{n-3}\setminus\widetilde{\Delta}$, $s_1,s_2,s_3,t_1,\ldots,t_{n-3}\in\mathbb{A}^n$, and $\widetilde{\Delta}$ is the union of the extended diagonal and the hyperplanes given by asking at least one $p_i$ to be equal to 0 or 1. Of course, if $n=2$ we do not consider the condition $f(1,1)=s_3^r$, and there is no $t_i$, and similarly for $n=1$. Notice that $\mathbb{A}_{r,n,\text{sm}}(m)$ is an open subscheme of the scheme $\mathbb{A}_{r,n}(m)$ defined in the same way but without asking $f$ to be smooth.
					\end{definition}
					\textbf{First we address the case $n\geq3$.} As in~\cite{Per22}, we denote by $\sigma_\infty$ the section at infinity $S\rightarrow\mathbb{P}^1_S$, and similarly for $\sigma_0$ and $\sigma_{[1:1]}$. Define the fibered category $\widetilde{\mathcal{H}}_{r,g,n}^{\text{far}}$ whose objects are given by pairs of an object
					\[
					(\pi:P\rightarrow S,L,i:L^{\otimes r}\hookrightarrow\mathcal{O}_P,\sigma_1,\ldots,\sigma_n,j_1,\ldots,j_n)
					\]
					in $\mathcal{H}_{r,g,n}^{\text{far}}$, plus an isomorphism
					\[
						\begin{tikzcd}
							\phi:(P,L,\sigma_1,\sigma_2,\sigma_3)\arrow[r] & (\mathbb{P}_S^1,\mathcal{O}(-d),\sigma_\infty,\sigma_0,\sigma_{[1:1]}).
						\end{tikzcd}
					\]
					The arrows in $\widetilde{\mathcal{H}}_{r,g,n}^{\text{far}}$ are arrows in $\mathcal{H}_{r,g,n}^{\text{far}}$ which preserve $\phi$.
					
					Now, we want to describe $\widetilde{\mathcal{H}}_{r,g,n}^{\text{far}}$. It is not hard to show that it is equivalent to a functor, see~\cite[Remark 1.4]{Per22}. Moreover, we have an action of the group scheme
					\[
						G_{r,g,n}:=\underline{\mathrm{Aut}}_k(\mathbb{P}^1_k,\mathcal{O}(-d),\sigma_\infty,\sigma_0,\sigma_{[1:1]})
					\]
					on $\widetilde{\mathcal{H}}_{r,g,n}^{\text{far}}$, by composing an isomorphism $\phi$ with an element of $G_{r,g,n}$. The following Proposition is an analogue of~\cite[Proposition 1.5]{Per22}.
					\begin{proposition}\label{prop:isoHgnfartilde}
						Let $n\geq3$. The following holds:
						\begin{enumerate}
							\item There exists an isomorphism $G_{r,g,n}\simeq\mathbb{G}_m/\mu_{d}$.
							\item The functor $\widetilde{\mathcal{H}}_{r,g,n}^{\text{far}}$ is naturally isomorphic to $\mathbb{A}_{r,n,\text{sm}}(rd)$.
							\item The action of $G_{r,g,n}$ on $\widetilde{\mathcal{H}}_{r,g,n}^{\text{far}}$ translates into the action of \/ $\mathbb{G}_m/\mu_{d}$ on $\mathbb{A}_{r,n,\text{sm}}(rd)$ described as follows. Let $a$ be in $\mathbb{G}_m$ and
							\[
								(f,p_1,\ldots,p_{n-3},s_1,s_2,s_3,t_1,\ldots,t_{n-3})
							\]
							an element of $\mathbb{A}_{r,n,\text{sm}}(rd)$. Then
							\begin{align*}
								&a\cdot(f,p_1,\ldots,p_{n-3},s_1,s_2,s_3,t_1,\ldots,t_{n-3})\\
								&=(a^{-rd}f,p_1,\ldots,p_{n-3},a^{-d}s_1,a^{-d}s_2,a^{-d}s_3,a^{-d}t_1,\ldots,a^{-d}t_{n-3}).
							\end{align*}
							This action extends naturally to the whole $\mathbb{A}_{r,n}(rd)$.
						\end{enumerate}
					\end{proposition}
					\begin{proof}
						The first statement is clear, since $\underline{\mathrm{Aut}}_k(\mathbb{P}^1_k,\mathcal{O}(-d))\simeq\mathrm{GL}_2/\mu_{d}$, where $\mu_{d}$ injects in $\mathrm{GL}_2$ diagonally, see~\cite[Theorem 4.1]{AV04}.
						
						Now, let $S$ be a $k$-scheme. Recall that a section of $\mathcal{O}_{\mathbb{P}^1_S}(d)$ can be thought as a homogeneous rational function of degree $d$ in two variables, with coefficients in $\mathcal{O}_S$. Clearly, if $V\subset S$ is an open subscheme and $U=\pi^{-1}V\subset\mathbb{P}_S^1$, it makes sense to evaluate a polynomial $f\in\Gamma(U,\mathcal{O}_{\mathbb{P}^1_S}(d))$ in a pair of sections of $\mathcal{O}_S(V)$. Let $\sigma:S\rightarrow\mathbb{P}^1_S$ be a section of $\pi:\mathbb{P}^1_S\rightarrow S$, with image contained in $\mathbb{A}^1_S$, where we have eliminated a fixed section at infinity. Then $\sigma$ corresponds to a global section $p$ of $\mathcal{O}_S$. Therefore, we have a map
						\begin{equation}\label{eq:isopol}
							\begin{tikzcd}
								\sigma^*\mathcal{O}_{\mathbb{P}^1_S}(d)\arrow[r] & \mathcal{O}_S
							\end{tikzcd}
						\end{equation}
						induced by $f\mapsto f(1,p)$, since locally the sections of $\sigma^*\mathcal{O}_{\mathbb{P}^1_S}(d)$ are restrictions of polynomials to $\sigma(S)$. For that reason, the morphism is surjective between invertible sheaves on $S$, hence an isomorphism. Let $i:\mathcal{O}_{\mathbb{P}_S^1}(-rd)\hookrightarrow\mathcal{O}_{\mathbb{P}^1_S}$ such that $(\mathbb{P}^1_S,\mathcal{O}_{\mathbb{P}^1_S}(-d),i)$ is an object of $\mathcal{H}'_{r,g}$. Then $i$ corresponds to a polynomial $f\in\mathrm{H}^0(\mathcal{O}_{\mathbb{P}^1_S}(rd))$. Moreover, let $j:\sigma^*\mathcal{O}_{\mathbb{P}^1_S}(-d)\rightarrow\mathcal{O}_S$ be such that $j^r=\sigma^*(i)$. Then we get a global section $t$ of $\mathcal{O}_S\simeq\sigma^*\mathcal{O}_{\mathbb{P}^1_S}(d)$ corresponding to $j$ under the isomorphism~\eqref{eq:isopol}. The relation $j^r=\sigma^*(i)$ becomes
						\begin{equation}\label{eq:condizionign}
							f(1,p)=t^r
						\end{equation}
						between global sections of $\mathcal{O}_S$.
						
						In our case, we apply the above reasoning for each section $\sigma_v$ and each $j_v$. Indeed, we can always assume $P=\mathbb{P}_S^1$ and that the first section is $\sigma_\infty$ using the isomorphism $\phi$, and we can use a similar argument as above for $\sigma_\infty$, as originally done in~\cite[Proposition 1.5]{Per22}. In this case the isomorphism
						\[
							\begin{tikzcd}
								\sigma_{\infty}^*\mathcal{O}_{\mathbb{P}^1_S}(d)\arrow[r] & \mathcal{O}_S
							\end{tikzcd}
						\]
						is induced by $f\mapsto f(0,1)$. Then for any other section we can apply the reasoning above, since they are disjoint from the now fixed section at infinity $\sigma_\infty$.
						Since to give a global section of $\mathcal{O}_S$ is the same as giving a morphism $S\rightarrow\mathbb{A}^1$, we get a map
						\[
						\begin{tikzcd}[row sep=small, column sep=small]
							\widetilde{\mathcal{H}}^{\mathrm{far}}_{r,g,n}\arrow[rr] & & \mathbb{A}_{sm}(rd)\times(\mathbb{A}^{n-3}\setminus\widetilde{\Delta})\times\mathbb{A}^n.
						\end{tikzcd}
						\]
						This map sends an object
						\[
							(\pi:P\rightarrow S,L,i:L^{\otimes r}\hookrightarrow\mathcal{O}_P,\sigma_1,\ldots,\sigma_n,j_1,\ldots,j_n,\phi)
						\]
						in $(f,p_1,\ldots,p_{n-3},s_1,s_2,s_3,t_1,\ldots,t_{n-3})$, where we have read everything in $\mathbb{P}_S^1$ using the isomorphism $\phi$ (also the last $n-3$ sections). Here $f$, $(p_1,\ldots,p_{n-3})$, and $(s_1,\ldots,t_{n-3})$ correspond respectively to $(\mathcal{O}(-d),i)$, $(\sigma_4,\ldots,\sigma_{n})$, and $(j_1,\ldots,j_n)$.
						
						By construction, the image is $\mathbb{A}_{r,n,\mathrm{sm}}(rd)$.
						Clearly, we have a quasi-inverse to that functor, given by performing the inverse construction of the one before, and choosing $\phi$ to be the identity of
						\[
						(\mathbb{P}^1_S,\mathcal{O}_{\mathbb{P}^1_S}(-d),\sigma_{\infty},\sigma_0,\sigma_{[1,1]}).
						\]
						In this way, we get
						\[
						\widetilde{\mathcal{H}}^{\mathrm{far}}_{r,g,n}\simeq\mathbb{A}_{r,n,\mathrm{sm}}(rd).
						\]
						This proves the second statement.
						The last one follows from the construction of the isomorphism.
					\end{proof}
					\begin{remark}
						The action of an element $a$ on the polynomial $f$ can be written also as
						\[
							a\cdot f(x,y)=a^{-rd}f(x,y)=f(a^{-1}x,a^{-1}y)=f(A^{-1}\cdot(x,y))
						\]
						where $A$ is the diagonal $2\times2$ matrix with coefficients $a$ on the diagonal. See the articles~\cite{AV04},~\cite{Per22}, and what follows to understand the connection with the cases $n=0,1$.
					\end{remark}
					\textbf{Now we address the case of $n=1$ and $n=2$.}
					We proceed in the exact same way. For this reason we give no details and proofs. Define the fibered category $\widetilde{\mathcal{H}}_{r,g,1}^{\text{far}}$ whose objects are given by pairs of an object $(\pi:P\rightarrow S,L,i:L^{\otimes r}\hookrightarrow\mathcal{O}_P,\sigma_1,j_1)$ in $\mathcal{H}_{r,g,1}$, plus an isomorphism
					\[
					\begin{tikzcd}
						\phi:(P,L,\sigma_1)\arrow[r] & (\mathbb{P}_S^1,\mathcal{O}(-d),\sigma_\infty).
					\end{tikzcd}
					\]
					The arrows in $\widetilde{\mathcal{H}}_{r,g,1}^{\text{far}}$ are arrows in $\mathcal{H}_{r,g,1}^{\text{far}}$ that preserve $\phi$. Similarly, we define the fibered category $\widetilde{\mathcal{H}}_{r,g,2}^{\text{far}}$ whose objects are given by pairs of an object
					\[
						(\pi:P\rightarrow S,L,i:L^{\otimes r}\hookrightarrow\mathcal{O}_P,\sigma_1,\sigma_2,j_1,j_2)
					\]
					in $\mathcal{H}_{r,g,2}$, plus an isomorphism
					\[
					\begin{tikzcd}
						\phi:(P,L,\sigma_1,\sigma_2)\arrow[r] & (\mathbb{P}_S^1,\mathcal{O}(-d),\sigma_\infty,\sigma_0).
					\end{tikzcd}
					\]
					The arrows in $\widetilde{\mathcal{H}}_{r,g,2}^{\text{far}}$ are arrows in $\mathcal{H}_{r,g,2}^{\text{far}}$ which preserve $\phi$.
					Then, both $\widetilde{\mathcal{H}}_{r,g,1}^{\text{far}}$ and $\widetilde{\mathcal{H}}_{r,g,2}^{\text{far}}$ are equivalent to functors, see~\cite[Remark 1.4]{Per22}.
					
					Moreover, we have actions of the group schemes
					\[
					G_{r,g,1}:=\underline{\mathrm{Aut}}_k(\mathbb{P}^1_k,\mathcal{O}(-d),\sigma_\infty)
					\]
					and
					\[
					G_{r,g,2}:=\underline{\mathrm{Aut}}_k(\mathbb{P}^1_k,\mathcal{O}(-d),\sigma_\infty,\sigma_0)
					\]
					on $\widetilde{\mathcal{H}}_{r,g,1}^{\text{far}}$ and $\widetilde{\mathcal{H}}_{r,g,2}^{\text{far}}$ respectively, given by composing an isomorphism $\phi$ with an element of those groups. Let $\mathrm{B}_2$ be the subgroup of $2\times2$ lower triangular matrices in $\mathrm{GL}_2$, and consider $\mathbb{G}_m\times\mathbb{G}_m\subset\mathrm{B}_2$ as the subgroup of the diagonal matrices.
					\begin{proposition}\label{prop:isoHgnfartilde12}
						The following holds:
						\begin{enumerate}
							\item There exist isomorphisms
							\begin{align*}
								&G_{r,g,1}\simeq \mathrm{B}_2/\mu_{d}\\
								&G_{r,g,2}\simeq(\mathbb{G}_m\times\mathbb{G}_m)/\mu_{d}
							\end{align*}
							where $\mu_{d}$ injects in $\mathbb{G}_m\times\mathbb{G}_m\subset B_2$ diagonally.
							\item The functors $\widetilde{\mathcal{H}}_{r,g,1}^{\text{far}}$ and $\widetilde{\mathcal{H}}_{r,g,2}^{\text{far}}$ are naturally isomorphic to $\mathbb{A}_{r,1,\text{sm}}(rd)$ and $\mathbb{A}_{r,2,\text{sm}}(rd)$ respectively.
							\item The action of \/ $G_{r,g,1}$ on \/ $\widetilde{\mathcal{H}}_{r,g,1}^{\text{far}}$ translates into the action of \/ $\mathrm{B}_2/\mu_{d}$ on $\mathbb{A}_{r,1,\text{sm}}(rd)$ described as follows. Let $A=\begin{bmatrix}
								a & 0\\
								b & c\\
							\end{bmatrix}$ be an element of \/ $\mathrm{B}_2$, and $(f,s_1)$ an element of \/ $\mathbb{A}_{r,1,\text{sm}}(rd)$. Then
							\[
							A\cdot(f,s_1)=(f(A^{-1}(x,y)),c^{-d}s_1).
							\]
							This action extends naturally to the whole $\mathbb{A}_{r,1}(rd)$.
							\item The action of \/ $G_{r,g,2}$ on $\widetilde{\mathcal{H}}_{r,g,2}^{\text{far}}$ translates into the action of \/ $(\mathbb{G}_m\times\mathbb{G}_m)/\mu_{d}$ on $\mathbb{A}_{r,2,\text{sm}}(rd)$ described as follows. Let $(a,c)$ be an element of \/ $\mathbb{G}_m\times\mathbb{G}_m$, and $(f,s_1,s_2)$ an element of \/ $\mathbb{A}_{r,2,\text{sm}}(rd)$. Then
							\[
							(a,c)\cdot(f(x,y),s_1,s_2)=(f(a^{-1}x,c^{-1}y),c^{-d}s_1,a^{-d}s_2).
							\]
							This action extends naturally to the whole $\mathbb{A}_{r,2}(rd)$.
							\qed
						\end{enumerate}
					\end{proposition}				
					Now, let $n\geq1$. Notice that we have a natural map $\widetilde{\mathcal{H}}_{r,g,n}^{\text{far}}\rightarrow\mathcal{H}_{r,g,n}^{\text{far}}$ given by forgetting the isomorphism $\phi$. The following Proposition is an analogue of~\cite[Proposition 1.6]{Per22}, and can be proved in the same way. For this reason, we omit the proof.
					\begin{proposition}\label{prop:torsor}
						The map
						\[
						\begin{tikzcd}
							\widetilde{\mathcal{H}}_{r,g,n}^{\text{far}}\arrow[r] & \mathcal{H}_{r,g,n}^{\text{far}}
						\end{tikzcd}
						\]
						is a $G_{r,g,n}$-torsor. In particular, $\mathcal{H}_{r,g,n}^{\text{far}}\simeq\left[\widetilde{\mathcal{H}}_{r,g,n}^{\text{far}}/G_{r,g,n}\right]$.\qed
					\end{proposition}
					From Propositions~\ref{prop:isoHgnfartilde},~\ref{prop:isoHgnfartilde12} and~\ref{prop:torsor} we get the following Corollaries.
					\begin{corollary}\label{cor:quotientclosedHgnfar1}
						There exists an isomorphism
						\[
						\mathcal{H}_{r,g,1}\simeq\left[\mathbb{A}_{r,1,\text{sm}}(rd)/(\mathrm{B}_2/\mu_{d})\right]
						\]
						where the action of \/ $\mathrm{B}_2/\mu_{d}$ on $\mathbb{A}_{r,1,\text{sm}}(rd)$ is given by
						\[
						A\cdot(f,s_1)=(f(A^{-1}(x,y)),c^{-d}s_1)
						\]
						for all \/ $\mathrm{B}_2\ni A=\begin{bmatrix}
							a & 0\\
							b & c\\
						\end{bmatrix}$ and $(f,s_1)\in\mathbb{A}_{r,1,\text{sm}}(rd)$.\qed
					\end{corollary}					
					\begin{corollary}\label{cor:quotientclosedHgnfar2}
						There exists an isomorphism
						\[
						\mathcal{H}_{r,g,2}^{\text{far}}\simeq\left[\mathbb{A}_{r,2,\text{sm}}(rd)/((\mathbb{G}_m\times\mathbb{G}_m)/\mu_{d})\right]
						\]
						where the action of \/ $(\mathbb{G}_m\times\mathbb{G}_m)/\mu_{d}$ on $\mathbb{A}_{r,2,\text{sm}}(rd)$ is given by
						\[
						(a,c)\cdot(f(x,y),s_1,s_2)=(f(a^{-1}x,c^{-1}y),c^{-d}s_1,a^{-d}s_2)
						\]
						for all \/ $\mathbb{G}_m\times\mathbb{G}_m\ni(a,c)$ and $(f,s_1,s_2)\in\mathbb{A}_{r,2,\text{sm}}(rd)$.\qed
					\end{corollary}
					\begin{corollary}\label{cor:quotientclosedHgnfar}
						Let $n\geq3$. There exists an isomorphism
						\[
						\mathcal{H}_{r,g,n}^{\text{far}}\simeq\left[\mathbb{A}_{r,n,\text{sm}}(rd)/(\mathbb{G}_m/\mu_{d})\right]
						\]
						where the action of \/ $\mathbb{G}_m/\mu_{d}$ on $\mathbb{A}_{r,n,\text{sm}}(rd)$ is given by
						\[
						a\cdot(f,p_1,\ldots,p_{n-3},s_1,\ldots,t_{n-3})=(a^{-rd}f,p_1,\ldots,p_{n-3},a^{-d}s_1,\ldots,a^{-d}t_{n-3})
						\]
						for all \/ $\mathbb{G}_m\ni a$ and $(f,p_1,\ldots,p_{n-3},s_1,\ldots,t_{n-3})\in\mathbb{A}_{r,n,\text{sm}}(rd)$. This action extends naturally to the whole $\mathbb{A}_{r,n}(rd)$.\qed
					\end{corollary}
					For the case of $n=0$ see~\cite[Theorem 4.1]{AV04}, while for the case $(r,n)=(2,1)$ this has already been done in~\cite[Propositions 1.5, 1.6]{Per22}.
					\subsection{Second presentation of $\mathcal{H}_{r,g,n}^{\mathrm{far}}$, for $n\leq rd+1$}\label{subsec:secondpresentation}
					In order to compute the Picard group of $\mathcal{H}_{r,g,n}^{\text{far}}$ for $n\leq rd+1$, we want to find a presentation of $\widetilde{\mathcal{H}}_{r,g,n}^{\text{far}}$ as a quotient stack of an open subscheme of a representation of $G_{r,g,n}$, by the action of the same group. First we consider the case $1\leq n\leq3$, and only after that we address the case $3\leq n\leq rd+1$. This is motivated by both expository and notational reasons.
					
					\textbf{Case $1\leq n\leq3$.}
					Let $f\in\mathbb{A}(rd)$, and let $a_0,\ldots,a_{rd}$ be its coefficients, so that
					\[
						f(x,y)=\sum_{i=0}^{rd}a_ix^{rd-i}y^i.
					\]
					Then, the equation $f(0,1)=s_1^r$ is equivalent to $a_{rd}=s_1^r$. Similarly, $f(1,0)=s_2^r$ is equivalent to $a_0=s_2^r$, and $f(1,1)=s_3^r$ is the same as $a_{rd-1}=s_3^r-s_2^r-\sum_{i=1}^{rd-2}a_i-s_1^r$, where we have used the previous relations. This induces an isomorphism
					\[
						\begin{tikzcd}
							h_n:\mathbb{A}_{r,n}(rd)\arrow[r,"\simeq"] & \mathbb{A}^{rd+1}
						\end{tikzcd}
					\]
					for $1\leq n\leq3$. For $n=3$, this sends $(f,s_1,s_2,s_3)$ to $(a_1,\ldots,a_{rd-2},s_1,s_2,s_3)$, and the inverse $g_3$ is induced by the above relations. Define $\mathbb{A}_{r,n,\text{sm}}^{rd+1}$ to be the open subscheme of $\mathbb{A}_{r,n}^{rd+1}:=\mathbb{A}^{rd+1}$ which is the image of $\mathbb{A}_{r,n,\text{sm}}(rd)$ under $h_n$.
					We have an action of $G_{r,g,n}$ on $\mathbb{A}_{r,n}^{rd+1}$ which makes the map $h_n$ equivariant, described as follows.
					
					For $n=1$, given $\mathrm{B}_2\ni A=\begin{bmatrix}
						a & 0\\
						b & c\\
					\end{bmatrix}$ and $(a_0,\ldots,a_{rd-1},s_1)\in\mathbb{A}_{r,1}^{rd+1}$ with associated polynomial $f$, the action is
					\[
					A\cdot(a_0,\ldots,a_{rd-1},s_1)=(a_0',\ldots,a_{rd-1}',c^{-d}s_1)
					\]
					where $a_i'$ are the coefficients of $f'$ defined as $f'(x,y)=f(A^{-1}(x,y))$.
				
					For $n=2$, given $(a,c)\in\mathbb{G}_m\times\mathbb{G}_m$ and $(a_1,\ldots,a_{rd-1},s_1,s_2)\in\mathbb{A}_{r,2}^{rd+1}$, the action is
					\[
						(a,c)\cdot(a_1,\ldots,a_{rd-1},s_1,s_2)=(a^{-(rd-1)}c^{-1}a_1,\ldots,a^{-1}c^{-(rd-1)}a_{rd-1},c^{-d}s_1,a^{-d}s_2).
					\]
					
					For $n=3$, $a\in\mathbb{G}_m$ and $(a_1,\ldots,a_{rd-2},s_1,s_2,s_3)\in\mathbb{A}_{3}^{rd+1}$, the action is
					\[
					a\cdot(a_1,\ldots,a_{rd-2},s_1,s_2)=(a^{-rd}a_1,\ldots,a^{-rd}a_{rd-2},a^{-d}s_1,a^{-d}s_2,a^{-d}s_3).
					\]
					
					We obtain the following.
					\begin{proposition}\label{prop:presentationopenHgnfar23}
						Let $n\leq3$. Then
						\[
							\mathcal{H}_{r,g,n}^{\text{far}}\simeq\left[\mathbb{A}_{r,n,\text{sm}}^{rd+1}/G_{r,g,n}\right]
						\]
						with the action described above.\qed
					\end{proposition}
					\textbf{Case $3\leq n\leq rd+1$.}
					The idea is the same, but with some technical difficulties. We want to translate inductively the relation $f(1,p_i)=t_i^r$ in a relation of one coefficient of $f$, of the form $a_j\lambda_j=\psi_j$ for some regular functions $\lambda_j$ and $\psi_j$ not depending on $a_j$, and to invert $\lambda_j$. In particular, this implies the rationality of $\mathbb{A}_{r,n}(rd)$ for $n\leq rd+1$.
					
					First, define $\mathbb{A}_{r,n}^{rd-2+n}$ to be the open subscheme of $\mathbb{A}^{rd-2+n}$ where elements
					\[
					(a_1,\ldots,a_{rd+1-n},p_1,\ldots,p_{n-3},s_1,s_2,s_3,t_{1},\ldots,t_{n-3})\in\mathbb{A}^{rd-2+n}
					\]
					satisfies $p_i\not=p_j$ for all $i\not=j$, and are different from 0 and 1. If $\widetilde{\Delta}\subset\mathbb{A}^{n-3}$ is the union of the extended diagonal and the hyperplanes given by asking at least one coordinate to be equal to 0 or 1, we have $\mathbb{A}_{r,n}^{rd-2+n}=\mathbb{A}^{rd+1-n}\times(\mathbb{A}^{n-3}\setminus\widetilde{\Delta})\times\mathbb{A}^n$.
					
					There is an action of $\mathbb{G}_m/\mu_{d}$ on $\mathbb{A}^{rd-2+n}$ as follows. Let $a\in\mathbb{G}_m$, and
					\[
						(a_1,\ldots,a_{rd+1-n},p_1,\ldots,p_{n-3},s_1,s_2,s_3,t_{1},\ldots,t_{n-3})\in\mathbb{A}^{rd-2+n}
					\]
					then
					\begin{align*}
						&a\cdot(a_1,\ldots,a_{rd+1-n},p_1,\ldots,p_{n-3},s_1,\ldots,t_{n-3})\\
						&=(a^{-rd}a_1,\ldots,a^{-rd}a_{rd+1-n},p_1,\ldots,p_{n-3},a^{-d}s_1,\ldots,a^{-d}t_{n-3}).
					\end{align*}
					Of course, if $n=rd+1$ there are no $a_i$ terms. Notice that $\mathbb{A}_{r,n}^{rd-2+n}$ is invariant under this action.
					\begin{lemma}\label{lem:rationality}
						Let $3\leq n\leq rd+1$. Then the projection $h_n:\mathbb{A}_{r,n}(rd)\rightarrow\mathbb{A}^{rd-2+n}$ given by
						\[
							\begin{tikzcd}
								(f,p_1,\ldots,p_{n-3},s_1,\ldots,t_{n-3})\arrow[r,mapsto,"h_n"] & (a_1,\ldots,a_{rd+1-n},p_1,\ldots,p_{n-3},s_1,\ldots,t_{n-3})
							\end{tikzcd}
						\]
						is a $(\mathbb{G}_m/\mu_{d})$-equivariant open immersion with image $\mathbb{A}_{r,n}^{rd-2+n}$.
						The $(\mathbb{G}_m/\mu_{d})$-equivariant inverse
						\[
							\begin{tikzcd}
								g_n:\mathbb{A}_{r,n}^{rd-2+n}\arrow[r] & \mathbb{A}_{r,n}(rd)
							\end{tikzcd}
						\]
						is of the form
						\begin{align*}
							&g_n(a_1,\ldots,a_{rd+1-n},p_1,\ldots,p_{n-3},s_1,\ldots,t_{n-3})\\
							&=(s_2^r,a_1,\ldots,a_{rd+1-n},\varphi^{(n)}_{n-3},\varphi^{(n)}_{n-4},\ldots,\varphi^{(n)}_{1},\varphi^{(n)}_{0},s_1^r,p_1,\ldots,p_{n-3},s_1,\ldots,t_{n-3})
						\end{align*}
						where the $\varphi_{i}^{(n)}$ are as follows.
						\begin{enumerate}[label=\textbf{\arabic*}.]
							\item For all $i$,
							\[
								\varphi_i^{(n)}\in k[a_1,\ldots,a_{rd+1-n},s_1,s_2,s_3,t_1,\ldots,t_{n-3}](p_1,\ldots,p_{n-3})
							\]
							and it is regular over $\mathbb{A}_{r,n}^{rd-2+n}$.
							\item Every monomial in $\varphi_i^{(n)}$ can be of two types. If $a_j$ with $1\leq j\leq rd+1-n$ appears in that monomial, it does with degree 1, and no other $a_v$ or one between $s_1,s_2,s_3\ldots,t_{n-3}$ appears in it. If no $a_j$ appears in that monomial, then it is of degree $r$ in the variables $s_1,s_2,s_3,t_1,\ldots,t_{n-3}$.

							Equivalently, $\varphi_i^{(n)}$ is homogeneous of degree $r$ in $a_1,\ldots,a_{rd+1-n}$ and $s_1,s_2,s_3,t_1,\ldots,t_{n-3}$, if it is assigned degree $r$ to the first set of variables and degree 1 to the second.
							\item We have,
							\[
								\varphi^{(n)}_{0}=s_3^r-s_2^r-\sum_{i=1}^{rd+1-n}a_i-\sum_{j=1}^{n-3}\varphi_j^{(n)}-s_1^r.
							\]
							\item For all $i<n-3$,
							\begin{align*}
								&\varphi_i^{(n)}(a_1,\ldots,a_{rd+1-n},p_1,\ldots,p_{n-3},s_1,s_2,s_3,t_1,\ldots,t_{n-3})\\
								&=\varphi_i^{(n-1)}(a_1,\ldots,a_{rd+1-n},\varphi_{n-3}^{(n)},p_1,\ldots,p_{n-4},s_1,s_2,s_3,t_1,\ldots,t_{n-4}).
							\end{align*}
						\end{enumerate}
					\end{lemma}
					\begin{proof}
						The fact that $h_n$ is a $\mathbb{G}_m/\mu_{d}$-equivariant immersion follows immediately from the structure of the map and the existence of $g_n$. Moreover, the equivariance of $g_n$ follows from property \textbf{2}. Therefore, it is enough to show the existence of such $\varphi_i^{(n)}$. We do this by induction on $n$, starting from $n=3$, proving at the same time the following claim. It is important to prove them simultaneously, since they follow each other for fixed $n$.
						
						\textbf{Claim:} under the assumptions of this Lemma, the relations of $\mathbb{A}_{r,n}(rd)$ in the space $\mathbb{A}(rd)\times(\mathbb{A}^{n-3}\setminus\widetilde{\Delta})\times\mathbb{A}^n$ reduce to
						\[
						a_0=s_2^r,\quad a_{rd}=s_1^r,\quad a_{i}\lambda_{rd-1-i}^{(n)}=\psi_{rd-1-i}^{(n)}\quad\text{for all }i\in\{rd+2-n,\ldots,rd-1\}
						\]
						where the following holds.
						\begin{enumerate}[label=\textbf{\alph*}.]
							\item For all $i\in\{0,\ldots,n-3\}$,
							\[
								\psi_i^{(n)}\in k[a_1,\ldots,a_{rd+1-n},s_1,s_2,s_3,t_1,\ldots,t_{n-3}](p_1,\ldots,p_{n-3})
							\]
							and
							\[
								\lambda_i^{(n)}\in k(p_1,\ldots,p_i).
							\]
							Moreover, they are regular on $\mathbb{A}_{r,n}^{rd-2+n}$.
							\item $\lambda_i^{(n)}$ is invertible in $\mathrm{H}^0(\mathbb{A}_{r,n}^{rd-2+n},\mathcal{O}_{\mathbb{A}_{r,n}^{rd-2+n}})$.
							\item $\psi_i^{(n)}$ is homogeneous of degree $r$ in $a_1,\ldots,a_{rd+1-n}$ and $s_1,\ldots,t_{n-3}$, if it is given degree $r$ to the first set of variables and degree 1 to the second.
							\item $\lambda_{0}^{(n)}\equiv1$ (so $\psi_{0}^{(n)}=\varphi_0^{(n)}$).
							\item For all $i<n-3$,
							\begin{align*}
								&\psi_i^{(n)}(a_1,\ldots,a_{rd+1-n},p_1,\ldots,p_{n-3},s_1,s_2,s_3,t_1,\ldots,t_{n-3})\\
								&=\psi_i^{(n-1)}(a_1,\ldots,a_{rd+1-n},\varphi_{n-3}^{(n)},p_1,\ldots,p_{n-4},s_1,s_2,s_3,t_1,\ldots,t_{n-4})
							\end{align*}
							where $\varphi_{n-3}^{(n)}=\psi_{n-3}^{(n)}/\lambda_{n-3}^{(n)}$. Moreover, $\lambda_i^{(n)}=\lambda_i^{(i+3)}$.
						\end{enumerate}
						The case $n=3$ is already been taken care of in Proposition~\ref{prop:presentationopenHgnfar23}, so we can proceed with the inductive step. First, notice that if the claim holds for $n$, then the Lemma for $n$ follows from the case $n-1$. Indeed, since $\lambda_{n-3}^{(n)}$ is invertible, we may define $\varphi_{n-3}^{(n)}=\psi_{n-3}^{(n)}/\lambda_{n-3}^{(n)}$, as in \textbf{e} of the claim. We can then define $\varphi_{i}^{(n)}$ for $i<n-3$ as defined in \textbf{4}, which is therefore automatically satisfied. Also \textbf{3} is automatic. Moreover, from \textbf{e} it follows that $\varphi_{i}^{(n)}=\psi_{i}^{(n)}/\lambda_{i}^{(n)}$. Therefore, \textbf{1} and \textbf{2} follows from \textbf{a} and \textbf{c} respectively. Finally, by construction, $g_n$ is indeed the inverse of $h_n$.
						
						We are left with proving that the claim for $n$ follows from the statements of this Lemma and the claim both for $n-1$. Writing the relation
						\[
							f(1,p_{n-3})=t_{n-3}^r
						\]
						in terms of the coefficients of $f$ and using the inductive step, we get
						\begin{align*}
							t_{n-3}^r=s_2^r&+a_1p_{n-3}+\ldots+a_{rd+2-n}p_{n-3}^{rd+2-n}\\
							&+\varphi^{(n-1)}_{n-4}p_{n-3}^{rd+3-n}+\ldots+\varphi^{(n-1)}_{0}p_{n-3}^{rd-1}+s_1^rp_{n-3}^{rd}.
						\end{align*}
						Isolating the coefficient $a_{rd+2-n}$ (keep in mind the $\varphi_i^{(n-1)}$ depend on that coefficient), we get
						\[
							a_{rd+2-n}\lambda_{n-3}^{(n)}=\psi_{n-3}^{(n)}
						\]
						where $\lambda_{n-3}^{(n)}=p_{n-3}^{rd+2-n}\widetilde{\lambda}_{n-3}^{(n)}$, with
						\begin{align*}
							\widetilde{\lambda}_{n-3}^{(n)}&=1+\sum_{i=rd+3-n}^{rd-1}(\text{coefficient of }a_{rd+2-n}\text{ in }\varphi_{rd-1-i}^{(n-1)})\cdot p_{n-3}^{i-(rd+2-n)}\\
							&=1+\sum_{i=1}^{n-3}(\text{coefficient of }a_{rd+2-n}\text{ in }\varphi_{n-3-i}^{(n-1)})\cdot p_{n-3}^{i}
						\end{align*}
						and $\psi_{n-3}^{(n)}$ is of the form
						\begin{align*}
							\psi_{n-3}^{(n)}=t_{n-3}^r&-s_2^r-\sum_{i=1}^{rd+2-n}a_ip_{n-3}^i\\
							&-\sum_{i=0}^{n-4}(\text{terms in }\varphi_{n-4-i}^{(n-1)}\text{ without }a_{rd+2-n})\cdot p_{n-3}^{rd+3-n+i}-s_1^rp_{n-3}^{rd}.
						\end{align*}
						Thanks to both \textbf{1} and \textbf{2} for $n-1$, the properties \textbf{a} and \textbf{c} hold for $n$. Moreover, by the inductive hypothesis, the same type of relation holds for all $a_{rd-1-i}$ with $i<n-3$, and we can extend the definition of $\psi_{i}^{(n)}$ and $\lambda_i^{(n)}$ imposing that \textbf{e} holds. We are left with proving that \textbf{b} holds for $i=n-3$, since for $i<n-3$ it follows by the inductive hypothesis, being $\lambda_i^{(n)}=\lambda_i^{(i+3)}$ (and \textbf{d} is automatic).
						
						We already know that $\lambda_{n-3}^{(n)}$ is a quotient of polynomials in the variables $p_j$. Consider now the $p_i$ as fixed coefficients for $i<n-3$. Then, by construction, $\widetilde{\lambda}_{n-3}^{(n)}$ is a non-zero polynomial of degree at most $n-3$ in $p_{n-3}$, with constant term 1. Moreover, $\lambda_{n-3}^{(n)}$ vanishes exactly where the relation $f(1,p_{n-3})=t_{n-3}^r$ does not influence the coefficient $a_{rd+2-n}$. Clearly, this happens if $p_{n-3}$ assumes one of the following values:
						\[
							0, 1, p_1,\ldots, p_{n-4}.
						\]
						Therefore, $1,p_1,\ldots,p_{n-4}$ are distinct roots of $\widetilde{\lambda}_{n-3}^{(n)}$, and by the above remark about the degree of that polynomial, they are the all and only roots (this incidently shows that $\widetilde{\lambda}_{n-3}^{(n)}$ has exactly degree $n-3$ in $p_{n-3}$). Therefore, $\lambda_{n-3}^{(n)}$ is always non-zero in $\mathbb{A}_{r,n}^{rd-2+n}$, hence invertible. This proves \textbf{b}, and concludes the proof.
					\end{proof}
					\begin{proposition}\label{prop:quotientopenHgnfar}
						Let $3\leq n\leq rd+1$. Let $\mathbb{A}^{rd-2+n}_{r,n,\text{sm}}$ be the image of $\mathbb{A}_{r,n,\text{sm}}(rd)$ under $h_n$. Then
						\[
						\mathcal{H}_{r,g,n}^{\text{far}}\simeq\left[\mathbb{A}_{r,n,\text{sm}}^{rd-2+n}/G_{r,g,n}\right]
						\]
						with the action described above.\qed
					\end{proposition}
					Everything we have done so far has been already done in~\cite{AV04} for $n=0$ and in~\cite{Per22} for $(r,n)=(2,1)$. Obviously, in these cases $\mathcal{H}_{r,g,n}^{\text{far}}$ is just $\mathcal{H}_{r,g,n}$. We report here the result of Arsie and Vistoli, since it will be used in what follows. Here $\mathbb{A}_{r,0,\text{sm}}^{rd+1}$ is simply $\mathbb{A}_{\text{sm}}(rd)$.
					\begin{proposition}\label{prop:quotientopenHgn01}
						Let $n=0$, then
						\[
							\mathcal{H}_{r,g}=\left[\mathbb{A}_{r,0,\text{sm}}^{rd+1}/G_{r,g,0}\right]
						\]
						with $G_{r,g,0}=\mathrm{GL}_2/\mu_{d}$, where $\mu_{d}$ injects diagonally in $\mathrm{GL}_2$. The action is given by
						\[
							A\cdot f(x,y)=f(A^{-1}(x,y)).
						\]
					\end{proposition}
					\begin{proof}
						See~\cite[Corollary 4.2]{AV04}.
					\end{proof}
				\subsection{The divisor $\Delta_{r,g,n}$}
					In this subsection we always assume $n\leq rd+1$.
					In order to use the description of $\mathcal{H}_{r,g,n}^\text{far}$ found in~\ref{prop:presentationopenHgnfar23}, in~\ref{prop:quotientopenHgnfar} and in~\ref{prop:quotientopenHgn01} to compute $\mathrm{Pic}(\mathcal{H}_{r,g,n}^{\text{far}})$, it is crucial to study the complement $\Delta_{r,g,n}$ of $\mathbb{A}_{r,n,\text{sm}}^{rd-2+n}$ in $\mathbb{A}_{r,n}^{rd-2+n}$ for $3\leq n\leq rd+1$, and the complement of $\mathbb{A}_{r,n,\text{sm}}^{rd+1}$ in $\mathbb{A}_{r,n}^{rd+1}$ for $n\leq3$. For now $\Delta_{r,g,n}$ is considered just as a topological space.
					
					At the end of the proof of~\cite[Theorem 5.1]{AV04} it is shown that $\Delta_{r,g,0}$ can be described as the zero locus of a homogeneous polynomial $\Phi_{r,g,0}$ of degree $2(rd-1)$ in the variables $a_0,\ldots,a_{rd}$, and that this polynomial is irreducible if the characteristic of the ground field $k$ does not divide $2rd$, which is our case. We consider $\Delta_{r,g,0}$ with that scheme structure.
					
					Now, we define the scheme structure of $\Delta_{r,g,n}$ via the following diagram with cartesian squares and $\mathbb{G}_m$-equivariant maps (where $\mathbb{G}_m$ is also seen as the group of invertible diagonal $2\times2$ matrices that are multiples of the identity).
					\[
						\begin{tikzcd}[column sep=16pt]
							\Delta_{r,g,n}\arrow[d]\arrow[r] & \Delta_{r,g,n-1}\arrow[d]\arrow[r] & ...\arrow[d,dashed]\arrow[r] & \Delta_{r,g,3}\arrow[d]\arrow[r] & \Delta_{r,g,2}\arrow[d]\arrow[r] & \Delta_{r,g,1}\arrow[d]\arrow[r] & \Delta_{r,g,0}\arrow[d]\\
							\mathbb{A}_{r,n}^{rd-2+n}\arrow[r,"q_{n,n-1}"'] & \mathbb{A}_{r,n-1}^{rd-3+n}\arrow[r] & ...\arrow[r] & \mathbb{A}_{r,3}^{rd+1}\arrow[r,"q_{3,2}"'] &\mathbb{A}_{r,2}^{rd+1}\arrow[r,"q_{2,1}"'] &\mathbb{A}_{r,1}^{rd+1}\arrow[r,"q_{1,0}"'] &\mathbb{A}_{r,0}^{rd+1}
						\end{tikzcd}
					\]
					The maps $q_{n,n-1}:\mathbb{A}_{r,n}^{rd-2+n}\rightarrow\mathbb{A}_{r,n-1}^{rd-3+n}$ are given simply by reconstructing $a_{rd+2-n}$ using $\varphi_{n-3}^{(n)}$, see Lemma~\ref{lem:rationality}, and forgetting the last section, and similarly for $q_{3,2}$, $q_{2,1}$ and $q_{1,0}$. Essentially, together they give a factorization of the map $g_n$ of Lemma~\ref{lem:rationality} composed with the map which forgets the sections.
					
					By construction, $\Delta_{r,g,n}$ is the zeros locus of a polynomial $\Phi_{r,g,n}$, thanks to the result mentioned before for $n=0$. It is easy to describe $\Phi_{r,g,n}$ in terms of $\Phi_{r,g,0}$. Indeed, it holds
					\begin{align}
						&\Phi_{r,g,1}(a_0,\ldots,a_{rd-1},s_1)=\Phi_{r,g,0}(a_0,\ldots,a_{rd-1},s_1^r)\\
						&\Phi_{r,g,2}(a_1,\ldots,a_{rd-1},s_1,s_2)=\Phi_{r,g,0}(s_2^r,a_1,\ldots,a_{rd-1},s_1^r)
					\end{align}
					and for $n\geq3$
					\begin{align}
						&\Phi_{r,g,n}(a_1,\ldots,a_{rd+1-n},p_1,\ldots,p_{n-3},s_1,s_2,s_3,t_1,\ldots,t_{n-3})\nonumber\\
						&=\Phi_{r,g,0}(s_2^r,a_1,\ldots,a_{rd+1-n},\varphi^{(n)}_{n-3},\varphi^{(n)}_{n-4},\ldots,\varphi^{(n)}_{1},\varphi^{(n)}_{0},s_1^r).
					\end{align}
					\begin{proposition}\label{prop:integralityDeltagn}
						The closed subscheme $\Delta_{r,g,n}$ is integral, or equivalently $\Phi_{r,g,n}$ is irreducible.
					\end{proposition}
					\begin{proof}
						We prove the statement by induction on $n$, with the case $n=0$ already done, and treating the cases $1\leq n\leq3$ and $4\leq n\leq rd+1$ separately.
						
						\textbf{Case $1\leq n\leq3$.}
						First we show that $\Delta_{r,g,n}$ is irreducible. Notice that a polynomial $h\in\mathbb{A}(rd)$ is in $\Delta_{r,g,0}$ if and only if, after extending the base field, there exist some polynomials $f,g$ of degree 1 and $rd-2$ respectively such that $h=f^2g$.
						Consider the affine scheme $X_3$ given by tuples $(f,g,s_1,s_2,s_3)\in\mathbb{A}(1)\times\mathbb{A}(rd-2)\times\mathbb{A}^3$ such that
						\[
							(f^2g)(0,1)=s_1^r,\quad (f^2g)(1,0)=s_2^r,\quad (f^2g)(1,1)=s_3^r.
						\]
						Write $f=ax+by$ and $g=\sum c_jx^{rd-2-j}y^j$. Then $X_3\subset\mathbb{A}^{rd+4}$ is defined by the equations
						\[
							b^2c_{rd-2}=s_1^r,\quad a^2c_0=s_2^r,\quad (a+b)^2(\sum c_j)=s_3^r
						\]
						which generate the ideal $I_3$ of $X_3$.
						In particular, the locus in $X_3$ where each $s_j$ is different from 0 is isomorphic to an open subscheme of $\mathbb{A}^{rd+1}$, hence it is irreducible. Now, suppose by contradiction that $X_3$ is not irreducible. From what we have noticed it follows that there exists a polynomial $f$ not in $I_3$ such that $fs_i\in I_3$ for some $i$, eventually passing to the algebraic closure of $k$. Notice that the generators of $I_3$ given above form a reduced Gr\"obner basis for $I_3$ with respect to the monomial ordering $s_3>s_2>s_1>c_{rd-2}>\ldots>c_0>b>a$. Therefore we may assume that $f$ has degree $\leq r-1$ in each $s_j$, and exactly $r-1$ in $s_i$. Writing $f=s_i^{r-1}f_1+f_2$ where $\deg_{s_i}f_2<r-1$, we get that $s_if=s_i^rf_1+s_if_2=f_3+s_if_2\in I_3$, where $\deg_{s_i}f_3=0$. Since necessarily $f_3+s_if_2=0$, we get a contradiction. It follows that $X_3$ is irreducible.
						
						Since the map $X_3\rightarrow\mathbb{A}_{r,3}^{rd+1}$ induced by sending $(f,g,s_1,s_2,s_3)$ in $(f^2g,s_1,s_2,s_3)$ has set-theoretic image $\Delta_{r,g,3}$ (up to extending the field), this is also irreducible. It follows immediately that $\Delta_{r,g,n}$ is irreducible for all $n\leq3$.
						
						Now, we show by induction that $\Delta_{r,g,n}$ is reduced. First, notice that when $1\leq n\leq3$ the map $q_{n,n-1}$ is a finite flat cover of degree $r$, and so is its base change $\Delta_{r,g,n}\rightarrow\Delta_{r,g,n-1}$.
						Assume that $\Delta_{r,g,n-1}$ is integral. Notice that $\Delta_{r,g,n}\rightarrow\Delta_{r,g,n-1}$ is also generically étale, since it ramifies over the locus where $a_{rd+1}$, $a_0$ or $\sum a_j$ vanishes, depending on $n$, and those sections are non-zero divisors in $\Delta_{r,g,n-1}$ since this is integral. Therefore, $\Delta_{r,g,n}$ is generically reduced, hence reduced, being an hypersurface.
						
						\textbf{Case $4\leq n\leq rd+1$.} Suppose we have already proved the statement for $n-1$, with $n\geq4$. Notice that the fiber of $q_{n,n-1}$ over a point with associated polynomial $f$ can be identified with a non-empty open subscheme of the curve in $\mathbb{A}^1\times\mathbb{A}^1$ given by the points $(p_{n-3},t_{n-3})$ which satisfy $f(1,p_{n-3})=t_{n-3}^r$. Therefore, by miracle flatness the map $q_{n,n-1}$ is flat. It follows that its base change $\Delta_{r,g,n}\rightarrow\Delta_{r,g,n-1}$ is flat, and in particular open. Now, it is easy to see that $t^r-f(1,p)$ is reducible as a polynomial in $p$ and $t$ only if every root of $f(1,p)$ in the algebraic closure of $k$ has multiplicity at least 2.
						
						We claim that the generic fiber is integral. Indeed, the affine plane curve defined by $t^r-f(1,p)$ is singular if and only if in the algebraic closure of $k$ there exists a root of $f(1,p)$ with multiplicity greater than 1 (recall that the characteristic of $k$ does not divide $r$).
						Since $\deg f=rd>2$ if $f$ is non-zero, it follows that the locus of $\Delta_{r,g,n-1}$ over which the fiber is integral is open and non-empty, hence dense. Notice that over that open substack $\Delta_{r,g,n}$ is integral, since there $q_{n,n-1}$ is a flat morphism with integral fibers and base. Since the complement maps on a proper closed subset of $\Delta_{r,g,n-1}$ and $q_{n,n-1}$ is open, there cannot be an irreducible component in that locus. It follows that $\Delta_{r,g,n}$ is irreducible. What we have done above shows also that $\Delta_{r,g,n}$ is generically reduced, hence reduced being an hypersurface.
					\end{proof}
			\section{Computation of $\mathrm{Pic}(\mathcal{H}_{r,g,n}^{\mathrm{far}})$ for $n\leq rd+1$}\label{sec:case<=rd+1}
			\subsection{Computation of $\mathrm{Pic}(\mathcal{H}_{r,g,1})$ and $\mathrm{Pic}(\mathcal{H}_{r,g,2}^{\mathrm{far}})$}\label{sec:case12}
			We use the description in Proposition~\ref{prop:presentationopenHgnfar23} to compute the Picard groups of $\mathcal{H}_{r,g,1}$ and $\mathcal{H}_{r,g,2}^{\text{far}}$. We will use the same technique as~\cite{AV04}. We mention that the case $(r,n)=(2,1)$ has been treated in~\cite{Per22}, where the integral Chow ring has been calculated. Finally, we will find a section of the restriction homomorphism $\mathrm{Pic}(\mathcal{H}_{r,g,2})\rightarrow\mathrm{Pic}(\mathcal{H}_{r,g,2}^{\text{far}})$. 
			
			Using the results in~\cite{EG98}, in~\cite{AV04} it is shown that if $G$ is an algebraic group over $k$, $V$ an $l$-dimensional representation of $G$, $X$ an open invariant subscheme of $V$, then the terms in the exact sequence
			\begin{equation}\label{eq:exactgeneralHgnfar}
				\begin{tikzcd}
					\mathrm{A}_{l-1}^G(V\setminus X)\arrow[r] & \mathrm{A}_G^1(V)\arrow[r] & \mathrm{A}_G^1(X)\arrow[r] & 0
				\end{tikzcd}
			\end{equation}
			satisfy $\mathrm{A}_G^1(V)\simeq\widehat{G}$, the group of characters of $G$, and $\mathrm{A}_G^1(X)\simeq\mathrm{Pic}_{\text{fun}}([X/G])$. See~\cite[Theorem 5.1]{AV04} for details.
			
			In our case, for $n=1$ and $n=2$, we have $G=G_{r,g,n}$, $V=\mathbb{A}_{r,n}^{rd+1}$ and $X=\mathbb{A}_{r,n,\text{sm}}^{rd+1}$, while $\mathrm{A}_{rd}^G(V\setminus X)$ is generated by the class of $\Delta_{r,g,n}$ thanks to Proposition~\ref{prop:integralityDeltagn}, so the exact sequence~\eqref{eq:exactgeneralHgnfar} reads
			\[
			\begin{tikzcd}
				{[\Delta_{r,g,n}]}\mathbb{Z}\arrow[r] & 	\mathrm{A}_{G_{r,g,n}}^1(\mathbb{A}_{r,n}^{rd+1})\simeq\widehat{G}_{r,g,n}\arrow[r] & \mathrm{Pic}(\mathcal{H}_{r,g,n}^{\text{far}})\arrow[r] & 0.
			\end{tikzcd}
			\]			
			Therefore, it is enough to compute $\widehat{G}_{r,g,n}$ and to study the action of $G_{r,g,n}$ on the ideal sheaf $I_{\Delta_{r,g,n}}\simeq\mathcal{O}_{\mathbb{A}_{r,n}^{rd+1}}$.
			\subsubsection{Computation of \/ $\mathrm{Pic}(\mathcal{H}_{r,g,1})$}\label{subsec:case1}
			We start with the case $n=1$. We remind the reader that in this case $\mathcal{H}_{r,g,1}^{\text{far}}$ is simply $\mathcal{H}_{r,g,1}$. Recall $G_{r,g,1}=\mathrm{B}_2/\mu_{d}$, where $\mu_{d}$ injects diagonally in the subgroup of lower triangular $2\times2$ matrices $B_2$. We have an exact sequence
			\[
			\begin{tikzcd}[row sep=small, column sep=small]
				0 \arrow[rr] & & \mathbb{G}_a\arrow[rr] & & \mathrm{B}_2\arrow[rr] & & \mathbb{G}_m\times\mathbb{G}_m\arrow[rr] & & 1\\
				& & b\arrow[rr,mapsto] & &
				\begin{bmatrix}
					1 & 0\\
					b & 1
				\end{bmatrix}
				\\
				& & & &
				\begin{bmatrix}
					a & 0\\
					b & c
				\end{bmatrix}
				\arrow[rr,mapsto] & & (a,c)
			\end{tikzcd}
			\]
			Moreover, there exists a section
			\[
			\mathbb{G}_m\times\mathbb{G}_m\ni(a,c)\mapsto
			\begin{bmatrix}
				a & 0\\
				0 & c\\
			\end{bmatrix}
			\in\mathrm{B}_2
			\]
			in particular $\mathrm{B}_2\simeq\mathbb{G}_a\rtimes(\mathbb{G}_m\times\mathbb{G}_m)$.
			It is well know that $\widehat{\mathbb{G}_a}=0$, hence the inclusion $\mathbb{G}_m\times\mathbb{G}_m\hookrightarrow\mathrm{B}_2$ induces an isomorphism
			\[
			\begin{tikzcd}
				\widehat{\mathrm{B}_2}\arrow[r,"\simeq"] & (\mathbb{G}_m\times\mathbb{G}_m)^\wedge\simeq\widehat{\mathbb{G}_m}\times\widehat{\mathbb{G}_m}\simeq\mathbb{Z}\oplus\mathbb{Z}.
			\end{tikzcd}
			\]
			The quotient by $\mu_{d}$ induces $\widehat{\mathrm{B}_2/\mu_{d}}\hookrightarrow\widehat{\mathrm{B}}_2$, with index $d$. Indeed, composing with the above isomorphism, it corresponds to the set of indices $(i,j)$ such that $d|(i+j)$.
			
			Now, we compute the class of $\Delta_{r,g,1}$ in $\widehat{G}_{r,g,1}$. Since it is invariant, there exist integers $m_1$ and $m_2$ such that, for all $(a,c)\in\mathbb{G}_m\times\mathbb{G}_m\subset\mathrm{B}_2$,
			\[
				\Phi_{r,g,1}((a,c)\cdot(a_0,\ldots,a_{rd-1},s_1))=a^{m_1}c^{m_2}\Phi_{r,g,1}(a_0,\ldots,a_{rd-1},s_1)
			\]
			and $m_1$, $m_2$ are the indices of $[\Delta_{r,g,1}]\mathbb{Z}$ in $\widehat{\mathbb{G}_m}\times\widehat{\mathbb{G}_m}$. On the other hand, if $A$ is the diagonal matrix with elements $a$ and $c$ respectively, by\cite[Lemma 5.4]{AV04} we have
			\begin{align*}
				\Phi_{r,g,1}((a,c)\cdot(a_0,\ldots,a_{rd-1},s_1))&=\Phi_{r,g,0}(A\cdot f(x,y))=(\det A)^{-rd(rd-1)}\Phi_{r,g,0}(f)\\
				&=a^{-rd(rd-1)}c^{-rd(rd-1)}\Phi_{r,g,1}(a_0,\ldots,a_{rd-1},s_1)
			\end{align*}
			hence $m_1=m_2=-rd(rd-1)$. We have proved the following Lemma.
			\begin{lemma}\label{lem:classDeltag1}
				The class of $\Delta_{r,g,1}$ in $\widehat{\mathbb{G}_m}\times\widehat{\mathbb{G}_m}\simeq\mathbb{Z}\oplus\mathbb{Z}$ is $(rd(rd-1),rd(rd-1))$.\qed
			\end{lemma}
			Recall that the inclusion $\widehat{G}_{r,g,1}\subset\widehat{\mathbb{G}_m}\times\widehat{\mathbb{G}_m}$ has image the set of indices $(i,j)$ such that $d|(i+j)$. Via the map $(i,j)\mapsto(i,i+j)$, we can identify it with $\mathbb{Z}\oplus d\mathbb{Z}\subset\mathbb{Z}\oplus\mathbb{Z}$. We then have to compute the quotient of $\mathbb{Z}\oplus d\mathbb{Z}$ by the subgroup
			\[
				\{(rd(rd-1)a,2rd(rd-1)a)\ |a\in\mathbb{Z}\}.
			\]
			The above is clearly equivalent to computing the quotient of $\mathbb{Z}\oplus\mathbb{Z}$ by
			\[
				K=\{r(rd-1)a\cdot(d,2)\ |a\in\mathbb{Z}\}.
			\]
			We have two cases, depending on the parity of $d$.
			
			Suppose first $2|d$. Then, we can write the subgroup $K$ as
				\[
					K=\{2r(rd-1)a\cdot(d/2,1)\ |a\in\mathbb{Z}\}
				\]
				and consider
				\[
				\begin{tikzcd}[row sep=small, column sep=small]
					\mathbb{Z}\oplus\mathbb{Z}\arrow[rr] & & \mathbb{Z}\oplus\mathbb{Z}\\
					(a,b)\arrow[rr,mapsto] & & (a-(d/2)b,b)
				\end{tikzcd}
				\]
				given by the matrix
				\[
				\begin{bmatrix}
					1 & -d/2\\
					0 & 1
				\end{bmatrix}.
				\]
				Since it is invertible as a matrix with integral coefficients, the map is an isomorphism of groups, which sends $K$ to $0\oplus2r(rd-1)\mathbb{Z}$. It follows that
				\[
					\mathrm{Pic}(\mathcal{H}_{r,g,1})\simeq\mathbb{Z}\oplus\mathbb{Z}/2r(rd-1)\mathbb{Z}.
				\]
			
			Now, suppose $2|(d+1)$. Consider the map
				\[
				\begin{tikzcd}[row sep=small, column sep=small]
					\mathbb{Z}\oplus\mathbb{Z}\arrow[rr] & & \mathbb{Z}\oplus\mathbb{Z}\\
					(a,b)\arrow[rr,mapsto] & & (2a-db,-a+(\frac{d+1}{2})b)
				\end{tikzcd}
				\]
				given by the matrix
				\[
				\begin{bmatrix}
					2 & -d\\
					-1 & \frac{d+1}{2}
				\end{bmatrix}
				\]
				which is again invertible as a matrix with integral coefficients. It follows that the map is an isomorphisms of groups, sending $K$ to $0\oplus r(rd-1)\mathbb{Z}$. Therefore,
				\[
				\mathrm{Pic}(\mathcal{H}_{r,g,1})\simeq\mathbb{Z}\oplus\mathbb{Z}/r(rd-1)\mathbb{Z}.
				\]
								
			We have proved the following Theorem.
			\begin{theorem}\label{thm:picardHg1}
				Suppose the base field $k$ has characteristic non dividing $2rd$.
				\begin{itemize}
					\item If $d$ is even, then \/ $\mathrm{Pic}(\mathcal{H}_{r,g,1})\simeq\mathbb{Z}\oplus\mathbb{Z}/2r(rd-1)\mathbb{Z}$.
					\item If $d$ is odd, then \/ $\mathrm{Pic}(\mathcal{H}_{r,g,1})\simeq\mathbb{Z}\oplus\mathbb{Z}/r(rd-1)\mathbb{Z}$.\qed
				\end{itemize}
			\end{theorem}
			\subsubsection{Computation of \/ $\mathrm{Pic}(\mathcal{H}_{r,g,2}^{\mathrm{far}})$}\label{subsec:case2}
			Let $n=2$. Notice that in the previous computation we reduced to study the action by $\mathbb{G}_m\times\mathbb{G}_m$ on $\mathbb{A}_{r,1}^{rd+1}$ instead of the one by $\mathrm{B}_2$. Since $G_{r,g,2}\simeq(\mathbb{G}_m\times\mathbb{G}_m)/\mu_{d}$ and we have a similar description of $\Phi_{r,g,2}$ in terms of $\Phi_{r,g,0}$, we can proceed in the exact same way as before. Hence, we get the same result, which proof is therefore omitted.
			\begin{proposition}\label{prop:picardHgnfar2}
				Suppose the base field $k$ has characteristic non dividing $2rd$.
				\begin{itemize}
					\item If $d$ is even, then \/ $\mathrm{Pic}(\mathcal{H}_{r,g,2}^{\text{far}})\simeq\mathbb{Z}\oplus\mathbb{Z}/2r(rd-1)\mathbb{Z}$.
					\item If $d$ is odd, then \/ $\mathrm{Pic}(\mathcal{H}_{r,g,2}^{\text{far}})\simeq\mathbb{Z}\oplus\mathbb{Z}/r(rd-1)\mathbb{Z}$.\qed
				\end{itemize}
			\end{proposition}
			Now, we have an exact sequence
			\[
				\begin{tikzcd}
					\oplus_{l=1}^{r-1}{[\mathcal{Z}_{r,g,2}^{1,2,l}]}\mathbb{Z}\arrow[r] & \mathrm{Pic}(\mathcal{H}_{r,g,2})\arrow[r] & \mathrm{Pic}(\mathcal{H}_{r,g,2}^\text{far})\arrow[r] & 0.
				\end{tikzcd}
			\]
			We want to find a section of the last map. We know that there exists a commutative diagram with $(\mathbb{G}_m\times\mathbb{G}_m)$-equivariant maps
			\[
			\begin{tikzcd}[row sep=normal, column sep=small]
				\mathbb{A}_{r,2,\text{sm}}(rd) & (f,s_1,s_2)\arrow[l,phantom,"\ni"]\arrow[rr,mapsto]\arrow[d,mapsto] & & (a_1,\ldots,a_{rd-1},s_1,s_2)\arrow[r,phantom,"\in"]\arrow[d,mapsto,"q_{2,1}"] & \mathbb{A}_{r,2,\text{sm}}^{rd+1}\\
				\mathbb{A}_{r,1,\text{sm}}(rd) & (f,s_1)\arrow[l,phantom,"\ni"]\arrow[rr,mapsto] & & (a_0=s_2^r,a_1\ldots,a_{rd-1},s_1)\arrow[r,phantom,"\in"] & \mathbb{A}_{r,1,\text{sm}}^{rd+1}.
			\end{tikzcd}
			\]
			Together with the fact that the computation of the Picard groups of $\mathcal{H}_{r,g,1}$ and $\mathcal{H}_{r,g,2}^{\text{far}}$ proceed in the exact same way, this gives a commutative diagram
				\begin{equation}\label{diag:isopic12far}
				\begin{tikzcd}
					{[\Delta_{r,g,1}]}\mathbb{Z}\arrow[r]\arrow[d,"\simeq"] & \widehat{G}_{r,g,1}\arrow[r]\arrow[d,"\simeq"] & \mathrm{Pic}(\mathcal{H}_{r,g,1})\arrow[r]\arrow[d] & 0\\
					{[\Delta_{r,g,2}]}\mathbb{Z}\arrow[r] & \widehat{G}_{r,g,2}\arrow[r] & \mathrm{Pic}(\mathcal{H}^{\mathrm{far}}_{r,g,2})\arrow[r] & 0\\
				\end{tikzcd}
			\end{equation}
			where $\mathrm{Pic}(\mathcal{H}_{r,g,1})\rightarrow\mathrm{Pic}(\mathcal{H}_{r,g,2}^{\mathrm{far}})$ comes from the map $\mathcal{H}_{r,g,2}^{\mathrm{far}}\rightarrow\mathcal{H}_{r,g,1}$ induced by forgetting the last section. It follows that the last map is also an isomorphism.
			\begin{lemma}\label{lem:comparisonpicardHg2farHg1}
				The composite
				\[
					\begin{tikzcd}[row sep=small, column sep=small]
						\mathcal{H}^{\mathrm{far}}_{r,g,2}\arrow[r,phantom,"\subset"] & \mathcal{H}_{r,g,2}\arrow[rr] & & \mathcal{H}_{r,g,1}
					\end{tikzcd}
				\]
				induces a commutative diagram
				\[
					\begin{tikzcd}[row sep=normal, column sep=small]
						\mathrm{Pic}(\mathcal{H}_{r,g,1})\arrow[rd]\arrow[rr,"\simeq"] & & \mathrm{Pic}(\mathcal{H}^{\mathrm{far}}_{r,g,2})\\
						& \mathrm{Pic}(\mathcal{H}_{r,g,2})\arrow[ur]
					\end{tikzcd}
				\]
				where the horizontal map is an isomorphism.
				In particular, the map
				\[
				\begin{tikzcd}
					\mathrm{Pic}(\mathcal{H}_{r,g,2})\arrow[r] & \mathrm{Pic}(\mathcal{H}_{r,g,2}^{\text{far}})
				\end{tikzcd}
				\]
				admits a section.\qed
			\end{lemma}
		\subsection{Computation of $\mathrm{Pic}(\mathcal{H}_{r,g,n}^{\mathrm{far}})$ for $3\leq n\leq rd+1$}\label{sec:3<=n<=2g+3}
			We compute the Picard group of $\mathcal{H}_{r,g,n}^{\text{far}}$ for $3\leq n\leq rd+1$ with the same method as in the previous section. However, in this case a section of $\mathrm{Pic}(\mathcal{H}_{r,g,n})\rightarrow\mathrm{Pic}(\mathcal{H}_{r,g,n}^{\text{far}})$ exists only when $d$ is even, as we will see in Section~\ref{sec:finalcomputation}.
			
			Recall that we have a presentation of $\mathcal{H}_{r,g,n}^{\text{far}}$ as a quotient stack of an open subscheme of a representation of $G_{r,g,n}\simeq\mathbb{G}_m/\mu_{d}$, see Proposition~\ref{prop:quotientopenHgnfar}. Moreover, $\widehat{G}_{r,g,n}\hookrightarrow\widehat{\mathbb{G}_m}\simeq\mathbb{Z}$ has image of index $d$. 
			
			Now, the complement of $\mathbb{A}^{rd-2+n}_{r,n}$ in $\mathbb{A}^{rd-2+n}$ is given by $\mathbb{A}^{rd+1-n}\times\widetilde{\Delta}\times\mathbb{A}^n$, and ${G}_{r,g,n}$ acts trivially on the ideal sheaves of its irreducible components. Therefore, from the exact sequence~\eqref{eq:exactgeneralHgnfar} with $V=\mathbb{A}^{rd-2+n}$ and $X=\mathbb{A}_{r,n}^{rd-2+n}$, we have that $\mathrm{A}_{G_{r,g,n}}^1(\mathbb{A}_{r,n}^{rd-2+n})\simeq\widehat{G}_{r,g,n}$.
			
			Now, by Proposition~\ref{prop:integralityDeltagn} we know that $\Delta_{r,g,n}$ is integral, hence by the usual exact sequence we get
			\begin{equation}\label{eq:exactHgnfar3<=n<=2g+3}
			\begin{tikzcd}
				{[\Delta_{r,g,n}]}\mathbb{Z}\arrow[r] & \mathrm{A}_{G_{r,g,n}}^1(\mathbb{A}_{r,n}^{rd-2+n})\simeq\widehat{G}_{r,g,n}\arrow[r] & \mathrm{Pic}(\mathcal{H}_{r,g,n}^{\text{far}})\arrow[r] & 0.
			\end{tikzcd}
			\end{equation}
			Now, $\Delta_{r,g,n}$ is the zero locus of the irreducible polynomial $\Phi_{r,g,n}$, which is of degree $2(rd-1)$ in the $a_i$ and $2r(rd-1)$ in $s_1,s_2,s_3,t_1,\ldots,t_{n-3}$, for $3\leq n\leq rd+1$. Moreover, for all $a\in\mathbb{G}_m$, by~\cite[Lemma 5.4]{AV04} and the description of the $(\mathbb{G}_m/\mu_{d})$-equivariant map $\Delta_{r,g,n}\rightarrow\Delta_{r,g,0}$, we have
			\begin{align*}
				&\Phi_{r,g,n}(a\cdot(a_1,\ldots,a_{rd+1-n},p_1,\ldots,p_{n-3},s_1,\ldots,t_{n-3}))\\
				&=\Phi_{r,g,0}(A\cdot f)=(\det A)^{-rd(rd-1)}\Phi_{r,g,0}(f)\\
				&=a^{-2rd(rd-1)}\Phi_{r,g,n}(a_1,\ldots,a_{rd+1-n},p_1,\ldots,p_{n-3},s_1,\ldots,t_{n-3})
			\end{align*}
			where $f$ is the polynomial corresponding to $(a_1,\ldots,t_{n-3})$ and $A$ is the diagonal matrix with coefficients $a$. This shows that the class of $\Delta_{r,g,n}$ in $\widehat{\mathbb{G}_m}$ is $2rd(rd-1)$, hence its class in $\widehat{G}_{r,g,n}\simeq\mathbb{Z}$ is $2r(rd-1)$. Using the exact sequence~\eqref{eq:exactHgnfar3<=n<=2g+3}, we have proved the following Proposition.
			\begin{proposition}\label{prop:picardHgnfar3<=n<=2g+3}
				Let $3\leq n\leq rd+1$. Then
				\[
					\mathrm{Pic}(\mathcal{H}_{r,g,n}^{\text{far}})\simeq\mathbb{Z}/2r(rd-1)\mathbb{Z}
				\]
				regardless of the parity of $d$.\qed
			\end{proposition}
			\begin{remark}\label{rmk:uniformity3<=n<=rd+1}
				Notice that the way we have computed the Picard group of $\mathcal{H}_{r,g,n}^{\text{far}}$ for $3\leq n\leq rd+1$ shows that
				\[
					\begin{tikzcd}
						\mathrm{Pic}(\mathcal{H}_{r,g,n-1}^{\text{far}})\arrow[r] & \mathrm{Pic}(\mathcal{H}_{r,g,n}^{\text{far}})
					\end{tikzcd}
				\]
				is an isomorphism for $4\leq n\leq rd+1$. In particular,
				\[
				\begin{tikzcd}
					\mathrm{Pic}(\mathcal{H}_{r,g,3}^{\text{far}})\arrow[r] & \mathrm{Pic}(\mathcal{H}_{r,g,n}^{\text{far}})
				\end{tikzcd}
				\]
				is an isomorphism for $3\leq n\leq rd+1$.
			\end{remark}
			We could have computed the Picard group of $\mathcal{H}_{r,g,n}^{\text{far}}$ in a slightly different way, which we show now for the sake of completeness. First, we rewrite Proposition~\ref{prop:quotientopenHgnfar} as follows. Notice that $G_{r,g,n}=\mathbb{G}_m/\mu_{d}$ acts on $\mathbb{A}_{r,n,\text{sm}}^{rd-2+n}$ trivially on the components $p_i$, by the formula $a\cdot a_i=a^{-rd}a_i$ on the components $a_i$, and by the formula $a\cdot s=a^{-d}s$ on the components $s_1,s_2,s_3,t_1,\ldots,t_{n-3}$. This action factors through the isomorphism
			\[
			\begin{tikzcd}[row sep=small, column sep=normal]
				\mathbb{G}_m/\mu_{d}\arrow[r,"\simeq"] & \mathbb{G}_m
			\end{tikzcd}
			\]
			defined by $\overline{a}\mapsto a^{-d}$.
			Therefore, we can rewrite Proposition~\ref{prop:quotientopenHgnfar} as follows.
			\begin{corollary}\label{cor:representationweightedprojectivestack3<=n<=2g+3}
				Let $3\leq n\leq rd+1$. Then
				\[
				\mathcal{H}_{r,g,n}^{\text{far}}\simeq((\mathbb{A}^{n-3}\setminus\widetilde{\Delta})\times\mathcal{P}(r^{rd+1-n},1^n))\setminus\overline{\Delta}_{r,g,n}
				\]
				where $\mathcal{P}(r^{rd+1-n},1^n)$ is the weighted projective stack where the first $rd+1-n$ variables have degree $r$ while the last $n$ have degree 1, and $\overline{\Delta}_{r,g,n}$ is the image of $\Delta_{r,g,n}$ under the quotient. In particular, for $n=rd+1$
				\[
					\mathcal{H}_{r,g,rd+1}^{\text{far}}\simeq((\mathbb{A}^{rd-2}\setminus\widetilde{\Delta})\times\mathbb{P}^{rd})\setminus\overline{\Delta}_{r,g,rd+1}
				\]
				and it is a scheme.\qed
			\end{corollary}
			We can use Corollary~\ref{cor:representationweightedprojectivestack3<=n<=2g+3} to compute the Picard group of $\mathcal{H}_{r,g,n}^{\text{far}}$ for $3\leq n\leq rd+1$. Indeed, set $V_{r,g,n}:=\widetilde{\Delta}\times\mathcal{P}(r^{rd+1-n},1^n)\ \cup\overline{\Delta}_{r,g,n}$, we have an exact sequence of integral Chow groups
				\[
				\begin{tikzcd}[column sep=18pt]
					\mathrm{CH}_*(V_{r,g,n})\arrow[r] & \mathrm{CH}_*(\mathbb{A}^{n-3}\times\mathcal{P}(r^{rd+1-n},1^n))\arrow[r] &\mathrm{CH}_*(\mathcal{H}_{r,g,n}^{\text{far}}) \arrow[r] & 0.
				\end{tikzcd}
				\]
				Now, $\mathrm{CH}_{*+n-3}(\mathbb{A}^{n-3}\times\mathcal{P}(r^{rd+1-n},1^n))\simeq\mathrm{CH}_*(\mathcal{P}(r^{rd+1-n},1^n))$
				hence taking the component in codimension 1 we get
				\[
				\begin{tikzcd}
					{[\overline{\Delta}_{r,g,n}]}\mathbb{Z}\arrow[r] & \mathbb{Z}\arrow[r] & \mathrm{Pic}(\mathcal{H}_{r,g,n}^{\text{far}}) \arrow[r] & 0
				\end{tikzcd}
				\]
				since the image of the class of $\widetilde{\Delta}\times\mathcal{P}(r^{rd+1-n},1^n)$ is trivial.
				Using the fact that the degree of $\Phi_{r,g,n}$ is $2r(rd-1)$, we get again that
				\[
				\mathrm{Pic}(\mathcal{H}_{r,g,n}^{\text{far}})\simeq\mathbb{Z}/2r(rd-1)\mathbb{Z}
				\]
				for $3\leq n\leq rd+1$. The first approach has the advantage of allowing us to compare $\mathrm{Pic}(\mathcal{H}_{r,g,n})$ with $\mathrm{Pic}(\mathcal{H}_{r,g})$, which we will use in Section~\ref{sec:finalcomputation} to reconstruct $\mathrm{Pic}(\mathcal{H}_{r,g,n})$ from the exact sequence~\eqref{eq:exactmain}.
		\section{Computation of $\mathrm{Pic}(\mathcal{H}_{r,g,n}^{\mathrm{far}})$ for $n\geq rd+1$}\label{sec:>=2g+3}
			For $n\geq3$, we know how to describe $\mathcal{H}_{r,g,n}^{\text{far}}$ as a quotient of a locally closed subscheme of an affine space by an action of $\mathbb{G}_m/\mu_{d}$, see Corollary~\ref{cor:quotientclosedHgnfar}. Moreover, for $n\leq rd+1$ we found another presentation of $\mathcal{H}_{r,g,n}^{\text{far}}$ as a quotient stack of an open subscheme of a presentation of $\mathbb{G}_m/\mu_{d}$, see Proposition~\ref{prop:quotientopenHgnfar}. In the case of $n>rd+1$ this is no longer doable. However, already for $n\geq rd+1$, we are dealing just with schemes, as we will see in moment.
			
			Precisely, for $n=rd+1$, we have
			\[
				\mathcal{H}_{r,g,rd+1}^{\text{far}}\simeq((\mathbb{A}^{rd-2}\setminus\widetilde{\Delta})\times\mathbb{P}^{rd})\setminus\overline{\Delta}_{r,g,rd+1}
			\]
			where $\overline{\Delta}_{r,g,rd+1}$ is the image of $\Delta_{r,g,rd+1}$, so again $\overline{\Delta}_{r,g,rd+1}=\mathrm{V}(\Phi_{r,g,rd+1})$, see Corollary~\ref{cor:representationweightedprojectivestack3<=n<=2g+3}.
			Recall we got this presentation using the open immersion $h_{rd+1}$, see Lemma~\ref{lem:rationality}. For $n=rd+1+m>rd+1$ we can still use $h_{rd+1}$ to get a mix of the two types of presentations, applying the identity on the other components. Precisely, we get that
			\[
				\mathcal{H}_{r,g,rd+1+m}^{\text{far}}=\left[\widetilde{\mathcal{H}}_{r,g,rd+1+m}^{\text{far}}/(\mathbb{G}_m/\mu_{d})\right]
			\]
			where $\widetilde{\mathcal{H}}_{r,g,rd+1+m}^{\text{far}}$ is isomorphic to the closed subscheme of
			\[
				((\mathbb{A}^{rd-2+m}\setminus\widetilde{\Delta})\times\mathbb{A}^{rd+1+m})\setminus\mathrm{V}(\Phi_{r,g,rd+1})
			\]
			defined by $m$ polynomials $f_i:=t_i^r-f(1,p_i)$ for $rd-1\leq i\leq rd+m-2$, with the usual notation, each of which is homogeneous of degree $r$ in the variables $s_1,s_2,s_3,t_1,\ldots,t_{rd-2},t_i$. Again, the action of $\mathbb{G}_m/\mu_{d}$ is of degree $-d$ on the second set of variables and does not affect the others, hence $\mathcal{H}_{r,g,rd+1+m}^{\text{far}}$ is the closed subvariety of
			\[
				((\mathbb{A}^{rd-2+m}\setminus\widetilde{\Delta})\times\mathbb{P}^{rd+m})\setminus\mathrm{V}(\Phi_{r,g,rd+1})
			\]
			defined by the same polynomials. Precisely, we have the following.
			\begin{proposition}\label{prop:thirdpresentation}
			Let $n\geq rd+1$, and $m=n-rd-1$. Then $\mathcal{H}_{r,g,n}^{\text{far}}$ is the closed subscheme of
			\[
				((\mathbb{A}^{n-3}\setminus\widetilde{\Delta})\times\mathbb{P}^{n-1})\setminus\mathrm{V}(\Phi_{r,g,rd+1})
			\]
			defined by $m$ polynomials
			\[
			f_i(p_1,\ldots,p_{rd-2},p_i,s_1,s_2,s_3,t_1,\ldots,t_{rd-2},t_i):=t_i^r-f(1,p_i)
			\]
			for $i\in\{rd-1,\ldots,n-3\}$, where $f=\sum_{i=0}^{rd}a_ix^{rd-i}y^i$ with $a_0=s_2^r$, $a_{rd}=s_1^r$, and
			\[
			a_i=\varphi^{(rd+1)}_{rd-1-i}(p_1,\ldots,p_{rd-2},s_1,s_2,s_3,t_1,\ldots,t_{rd-2})
			\]
			for \/ $1\leq i\leq rd-1$.\qed
			\end{proposition}
			We use this description to compute the Picard group of $\mathcal{H}_{r,g,n}^{\text{far}}$ for $n>rd+1$.
			We will need the following two lemmas.
			\begin{lemma}\label{lem:smoothnessforgettingcurve}
				For all $n\geq0$, consider the map
				\[
					\begin{tikzcd}
						\mathcal{H}_{r,g,n}^{\text{far}}\arrow[r] & \mathcal{M}_{0,n}
					\end{tikzcd}
				\]
				which forgets the smooth curve and considers just the sections to the Brauer-Severi scheme. Then, that map is smooth with geometrically connected fibers.
			\end{lemma}
			\begin{proof}
				Using Corollary~\ref{cor:representationweightedprojectivestack3<=n<=2g+3}, in the case of $3\leq n\leq rd+1$ the map corresponds to the projection
				\[
				\begin{tikzcd}
					(\mathbb{A}^{n-3}\setminus\widetilde{\Delta})\times\mathcal{P}(r^{rd+1-n},1^n)\arrow[r]& \mathbb{A}^{n-3}\setminus\widetilde{\Delta}
				\end{tikzcd}
				\]
				which is smooth. Hence it is smooth for $n\leq3$ too. Now, let $n\geq2$ and consider the $n$-fold fibered product $\mathcal{M}_{0,1}^n:=\mathcal{M}_{0,1}\times_{\mathcal{M}_{0,0}}\ldots\times_{\mathcal{M}_{0,0}}\mathcal{M}_{0,1}$. Then, we have a cartesian diagram
				\[
					\begin{tikzcd}[row sep=small,column sep=small]
						\mathcal{H}_{r,g,n}^{\text{far}}\arrow[dd]\arrow[rr] && \mathcal{M}_{0,n}\arrow[dd]\\
						& \square\\
						\mathcal{C}_{r,g}^n\arrow[rr] && \mathcal{M}_{0,1}^n
					\end{tikzcd}
				\]
				Now, we know that $\mathcal{C}_{r,g}\simeq\mathcal{H}_{r,g,1}\rightarrow\mathcal{M}_{0,1}$ is smooth with geometrically connected fibers, hence also $\mathcal{C}_{r,g}^n\rightarrow\mathcal{M}_{0,1}^n$ is. The cartesian diagram above concludes.
			\end{proof}
			\begin{lemma}\label{lem:picardrelativecompleteintersection}
			Let $S$ be an integral scheme over $k$, $X$ a closed subscheme in $S\times_k \mathbb{P}^{n}_k$. Suppose that the first projection $q_1:X\rightarrow S$ is surjective, proper, and flat. Suppose moreover that the geometric fibers of $q_1$ are integral complete intersections of dimension $\geq3$. Then, the restriction map
			\[
			\begin{tikzcd}
				\mathrm{Pic}(S)\times\mathrm{Pic}(\mathbb{P}_k^n)\arrow[r] & \mathrm{Pic}(X)
			\end{tikzcd}
			\]
			is an isomorphism.
			\end{lemma}
			\begin{proof}
			Since every geometric fiber of $q_1$ is an integral complete intersection, we can apply the version of the Grothendieck-Lefschetz Theorem found in~\cite[Corollaire 3.7, exposé XII, pag.121]{SGA} to get that for all $s\in S$ the restriction map $\mathrm{Pic}(\mathbb{P}^n_{k(s)})\rightarrow\mathrm{Pic}(X_s)$ is an isomorphism. Let $L$ be an invertible sheaf on $X$. For all $s\in S$ there exists an integer $m_s$ such that $L_s$ is the restriction of $\mathcal{O}_{\mathbb{P}^n_{k(s)}}(m_s)$ to $X_s$. Since the map $q_1$ is flat, the function $s\mapsto m_s$ is constant, say $m\in\mathbb{Z}$. Let $q_2:X\rightarrow\mathbb{P}^n$ be the second projection. It follows that $L\otimes q_2^*\mathcal{O}_{\mathbb{P}^n}(-m)$ is trivial on each fiber of $q_1$. Since $q_1$ is proper, flat, with geometrically integral fibers, and the base $S$ is integral, by Grauert there exists an invertible sheaf $M$ on $S$ such that $L\otimes q_2^*\mathcal{O}_{\mathbb{P}^n}(-m)\simeq q_1^*M$. This implies that $\mathrm{Pic}(S)\times\mathrm{Pic}(\mathbb{P}^n)\rightarrow\mathrm{Pic}(X)$ is surjective. The injectivity is immediate. Indeed, suppose $q_1^*M\otimes q_2^*N$ is isomorphic to the structure sheaf of $X$, let $s:\Spec k(s)\rightarrow S$ be a point and $s_X:X_s\rightarrow X$ its base change. Then $s_X^*q_1^*M\simeq\mathcal{O}_{X_s}$, while $s_X^*q_2^*$ is an isomorphism. Thus $N$ is trivial. Finally, ${q_1}_*q_1^*M\simeq M$ by the projection formula, hence $M$ is also trivial.
			\end{proof}
			\begin{proposition}\label{prop:surjectivepicard}
			Let $n\geq rd+1$. Then, the restriction map
			\[
				\begin{tikzcd}
					\mathbb{Z}\simeq\mathrm{Pic}(\mathbb{P}^{n-1})\simeq\mathrm{Pic}((\mathbb{A}^{n-3}\setminus\widetilde{\Delta})\times\mathbb{P}^{n-1})\arrow[r] & \mathrm{Pic}(\mathcal{H}_{r,g,n}^{\text{far}})
				\end{tikzcd}
			\]
			is surjective, and induces an isomorphism $\mathrm{Pic}(\mathcal{H}_{r,g,n}^{\text{far}})\simeq\mathbb{Z}/2r(rd-1)\mathbb{Z}$.
			\end{proposition}
			\begin{proof}
			Write $n=rd+1+m$. Let $\widehat{\mathcal{H}}_{r,g,n}^{\text{far}}\subset(\mathbb{A}^{n-3}\setminus\widetilde{\Delta})\times\mathbb{P}^{n-1}$ be the closed subscheme defined by the same polynomials $f_i$, so that $\mathcal{H}_{r,g,n}^{\text{far}}=\widehat{\mathcal{H}}_{r,g,n}^{\text{far}}\setminus\mathrm{V}(\Phi_{r,g,rd+1})$. Then $\widehat{\mathcal{H}}_{r,g,n}^{\text{far}}$ is a complete intersection of $m$ hypersurfaces. This follows easily from the particular form of these polynomials, since the variables $p_j$ and $t_j$ with $i\not=j>rd+1$ do not appear in $f_i=t_i^r-f(1,p_i)$. For the same reason, the geometric fibers of the first projection $q_1:\widehat{\mathcal{H}}_{r,g,n}^{\text{far}}\rightarrow\mathbb{A}^{n-3}\setminus\widetilde{\Delta}$ are complete intersections. Notice that the restriction of $q_1$ to $\mathcal{H}_{r,g,n}^{\text{far}}$ coincides with the map
			\[
				\begin{tikzcd}
					\mathcal{H}_{r,g,n}^{\text{far}}\arrow[r] & \mathcal{M}_{0,n}
				\end{tikzcd}
			\]
			which forgets the smooth curve, and considers just the sections to the Brauer-Severi scheme. By Lemma~\ref{lem:smoothnessforgettingcurve} we know that it is smooth with geometrically connected fibers.
			In particular, the fibers of $q_1$ restricted to $\mathcal{H}_{r,g,n}^{\text{far}}$ are geometrically integral. Now, we want to show that the whole fibers of $q_1$ are geometrically integral.
			
			First of all, recall that the fiber of $q_1$ over a point $p=(p_1,\ldots,p_{n-3})\in\mathbb{A}^{n-3}\setminus\widetilde{\Delta}$ is defined by the $m$ equations $f_i=0$, where $f_i=t_i^r-f(1,p_i)$. More explicitly, in terms of the immersion $h_{rd+1}$ the relation $f_i$ is given by
			\[
				t_i^r-\left(s_2^r+\sum_{j=0}^{rd-2}\varphi_{rd-2-j}^{(rd+1)}p_i^{j+1}+s_1^rp_i^{rd}\right)
			\]
			where we regard $p_1,\ldots,p_{n-3}$ as constants (since we are looking at a fiber), and the $\varphi^{(rd+1)}_{rd-2-j}$ come from Lemma~\ref{lem:rationality}, so they do not involve the $m$ variables $t_{rd-1},\ldots,t_{n-3}$.
			Moreover, notice that the polynomials $f_{n-3},\ldots,f_{rd-1}$ form a reduced Gr\"obner base of the ideal $I_m=(f_{n-3},\ldots,f_{rd-1})$ with respect to the lexical monomial ordering given by
			\[
				t_{n-3}>t_{n-4}>\ldots>t_1>s_3>s_2>s_1.
			\]
			Therefore a polynomial is in $I_m$ if and only if it can be reduced to 0 using the division algorithm.
			
			Now we show that every geometric fiber of $q_1$ is irreducible. Indeed, we have already shown that the restriction of $q_1$ to the complement of $\mathrm{V}(\Phi_{r,g,rd+1})$ has geometrically integral fibers. Therefore, if a geometric fiber of $q_1$ is not irreducible then there exists a polynomial $h$ not in $I_m$ such that $h\cdot\Phi_{r,g,rd+1}\in I_m$. Thanks to the properties of the Gr\"obner bases, a non-zero polynomial whose degree in each $t_{rd-1},\ldots,t_{n-3}$ is $\leq r-1$ cannot be in $I_m$. Since we may assume that this property holds for $h$, and we know that $\Phi_{r,g,rd+1}$ does not depend on $t_{rd-1},\ldots,t_{n-3}$, we get that also $h\cdot\Phi_{r,g,rd+1}$ has degree $\leq r-1$ in those variables. We have reached the desired contradiction, hence every geometric fiber is irreducible.
			In particular, we have proved that every geometric fiber is generically reduced, hence reduced being a complete intersection and so Cohen-Macaulay. Therefore, every fiber of $q_1$ is geometrically integral.
			
			Moreover, every fiber is $rd$-dimensional. Again using the fact that $\widehat{\mathcal{H}}_{r,g,n}^{\text{far}}$ is Cohen-Macaulay, by miracle flatness it follows that $q_1$ is flat. Finally, it is clear that $q_1$ is proper. Therefore, we can apply Lemma~\ref{lem:picardrelativecompleteintersection}, obtaining that
			\[
				\begin{tikzcd}
					\mathbb{Z}\simeq\mathrm{Pic}(\mathbb{P}^{n-1})\simeq\mathrm{Pic}(\mathbb{A}^{n-3}\setminus\widetilde{\Delta})\times\mathrm{Pic}(\mathbb{P}^{n-1})\arrow[r] & \mathrm{Pic}(\widehat{\mathcal{H}}_{r,g,n}^{\text{far}})
				\end{tikzcd}
			\]
			is an isomorphism.
			Composing with the homomorphism
			\[
				\begin{tikzcd}
					\mathrm{Pic}(\widehat{\mathcal{H}}_{r,g,n}^{\text{far}})\arrow[r] & \mathrm{Pic}(\mathcal{H}_{r,g,n}^{\text{far}})
				\end{tikzcd}
			\]
			which is surjective since the second scheme is smooth, we get the first part of the Proposition. Notice that the kernel of the last map is generated by the invertible sheaf $L$ associated to $\mathrm{V}(\Phi_{r,g,rd+1})$. On the fibers over $\mathbb{A}^{n-3}\setminus\widetilde{\Delta}$, $L$ is the structure sheaf twisted by $2r(rd-1)$. This (and the proof of Lemma~\ref{lem:picardrelativecompleteintersection}) shows that the invertible sheaf associated to $\mathrm{V}(\Phi_{r,g,rd+1})$ is the image of $2r(rd-1)$ under $\mathbb{Z}\rightarrow\mathrm{Pic}(\widehat{\mathcal{H}}_{r,g,n}^{\text{far}})$. Therefore, $\mathrm{Pic}(\mathcal{H}_{r,g,n}^{\text{far}})\simeq\mathbb{Z}/2r(rd-1)\mathbb{Z}$.
		\end{proof}
		\begin{remark}\label{rmk:integralgeometricfibers}
			Let $\widehat{\mathcal{H}}_{r,g,n}^{\text{far}}\subset(\mathbb{A}^{n-3}\setminus\widetilde{\Delta})\times\mathbb{P}^{n-1}$ be the closed subscheme defined by the polynomials $f_i$, so that $\mathcal{H}_{r,g,n}=\widehat{\mathcal{H}}_{r,g,n}^{\text{far}}\setminus\mathrm{V}(\Phi_{r,g,rd+1})$. The proof of the above Proposition shows that the projection from $\widehat{\mathcal{H}}_{r,g,n}^{\text{far}}$ to $\mathcal{M}_{0,n}$ is flat, proper, and has integral geometric fibers. This will be used in Lemma~\ref{lem:invertibleglobalsectionsHgnfarn>=2g+3}.
		\end{remark}
		Now, Proposition~\ref{prop:surjectivepicard} has a useful Corollary, which we will use in Section~\ref{sec:finalcomputation}.
		\begin{corollary}\label{cor:injectivepicard}
			Let $n\geq rd+1$. The map
			\[
				\begin{tikzcd}
					q_{n+1,n}:\mathcal{H}^{\text{far}}_{r,g,n+1}\arrow[r] & \mathcal{H}^{\text{far}}_{r,g,n}
				\end{tikzcd}
			\]
			given by forgetting the last section, induces an isomorphism
			\[
				\begin{tikzcd}
					q_{n+1,n}^*:\mathrm{Pic}(\mathcal{H}^{\text{far}}_{r,g,n})\arrow[r,"\simeq"] & \mathrm{Pic}(\mathcal{H}^{\text{far}}_{r,g,n+1}).
				\end{tikzcd}
			\]
		\end{corollary}
		\begin{proof}
			Notice that it is enough to show that $q_{n,rd+1}:\mathcal{H}^{\text{far}}_{r,g,n}\rightarrow\mathcal{H}^{\text{far}}_{r,g,rd+1}$ induces an isomorphism $q_{n,rd+1}^*:\mathrm{Pic}(\mathcal{H}^{\text{far}}_{r,g,rd+1})\rightarrow\mathrm{Pic}(\mathcal{H}^{\text{far}}_{r,g,n})$ for all $n\geq rd+1$. Now, let $s$ be a geometric point of $\mathcal{M}_{0,n}$, and consider the commutative diagram with cartesian squares
			\[
				\begin{tikzcd}
					(\mathcal{H}_{r,g,n}^{\text{far}})_{s}\arrow[d]\arrow[r] & (\mathcal{H}_{r,g,rd+1}^{\text{far}})_{s'}\arrow[r]\arrow[d] & \Spec k(s)\arrow[d,"s"']\arrow[dd,bend left=60,"s'"]\\
					\mathcal{H}_{r,g,n}^{\text{far}}\arrow[dr]\arrow[r] & X\arrow[r]\arrow[d] & \mathcal{M}_{0,n}\arrow[d]\\
					& \mathcal{H}_{r,g,rd+1}^{\text{far}}\arrow[r] & \mathcal{M}_{0,rd+1}
				\end{tikzcd}
			\]
			Proposition~\ref{prop:surjectivepicard} shows that the map $\mathrm{Pic}(\mathcal{H}_{r,g,n}^{\text{far}})\rightarrow\mathrm{Pic}((\mathcal{H}_{r,g,n}^{\text{far}})_s)$ is an isomorphism, and similarly for $rd+1$. To conclude it is enough to show that
			\[
				\begin{tikzcd}
					\mathrm{Pic}((\mathcal{H}^{\text{far}}_{r,g,rd+1})_{s'})\arrow[r] & \mathrm{Pic}((\mathcal{H}^{\text{far}}_{r,g,n})_{s})
				\end{tikzcd}
			\]
			is an isomorphism. Notice that it is induced by the projection
			\[
				\begin{tikzcd}[row sep=small, column sep=small]
					\mathbb{P}^{n-1}_{k(s)}\arrow[r,phantom,"\supset"] & (\mathcal{H}_{r,g,n}^{\text{far}})_s\arrow[rr] & & \mathbb{P}^{rd}_{k(s)}
				\end{tikzcd}
			\]
			given by $(s_1,s_2,s_3,t_1,\ldots,t_{n-3})\mapsto(s_1,s_2,s_3,t_1,\ldots,t_{rd-2})$. Therefore, it is clear that the pullback of an hyperplane of $\mathbb{P}^{rd}_{k(s)}$ is the restriction of an hyperplane of $\mathbb{P}^{n-1}_{k(s)}$ to $(\mathcal{H}_{r,g,n}^{\text{far}})_s$. It follows that the homomorphism $\mathrm{Pic}((\mathcal{H}^{\text{far}}_{r,g,rd+1})_{s'})\rightarrow\mathrm{Pic}((\mathcal{H}^{\text{far}}_{r,g,n})_{s})$ is surjective, hence an isomorphism.
		\end{proof}
		\section{Computation of $\mathrm{Pic}(\mathcal{H}_{r,g,n})$}\label{sec:finalcomputation}
		The only thing left to prove is that the first map in the right exact sequence~\eqref{eq:exactmain} is injective, and then to use the exact sequence to reconstruct the Picard group of $\mathcal{H}_{r,g,n}$.
		\subsection{Injectivity of the first map}
		This amounts to proving that the relations given in Corollary~\ref{cor:relationsZgnij} are actually the only relations between the classes $[\mathcal{Z}_{r,g,n}^{i,j,l}]$. To show it, we compute the invertible global sections of $\mathcal{H}_{r,g,n}^{\text{far}}$ for $n\geq rd+1$.
		Notice that there are non-constant invertible global sections over $\mathcal{H}_{r,g,n}^{\text{far}}$, for example those induced by the non trivial relations between the $[\mathcal{Z}_{r,g,n}^{i,j,l}]$ seen in Corollary~\ref{cor:relationsZgnij}. However, we show that for $n\geq rd+1$ these are essentially the only invertible global sections.
		\begin{lemma}\label{lem:invertibleglobalsectionsHgnfarn>=2g+3}
			Let $n\geq rd+1$. Recall that $\mathcal{H}_{r,g,n}^{\text{far}}$ is isomorphic to the closed subscheme of $((\mathbb{A}^{n-3}\setminus\widetilde{\Delta})\times\mathbb{P}^{n-1})\setminus V(\Phi_{r,g,rd+1})$ defined by the equations $f_i$, as in Proposition~\ref{prop:thirdpresentation}. Let $p_1,\ldots,p_{n-3}$ be the coordinates in the first factor. Then, the elements in $\mathcal{O}^*(\mathcal{H}_{r,g,n}^{\text{far}})$ are of the form
			\[
			\beta\cdot\prod_{i<j}(p_i-p_j)^{a_{i,j}}\prod_{i=1}^{n-3}p_i^{b_i}\prod_{i=1}^{n-3}(p_i-1)^{c_i}
			\]
			with $\beta\in k^*$ and $a_{i,j},b_i,c_i\in\mathbb{Z}$. In particular, the multiplicative subgroup $Q_{r,g,n}$ of invertible global sections with coefficient $\beta=1$ is isomorphic to $\mathbb{Z}^{n(n-3)/2}$.
		\end{lemma}
		\begin{proof}
			Let $\widehat{\mathcal{H}}_{r,g,n}^{\text{far}}\subset(\mathbb{A}^{n-3}\setminus\widetilde{\Delta})\times\mathbb{P}^{n-1}$ be the closed subscheme defined by the same polynomials $f_i$, see Remark~\ref{rmk:integralgeometricfibers} and the proof of Proposition~\ref{prop:surjectivepicard}. It is clear that the invertible sections of $\mathcal{H}_{r,g,n}^{\text{far}}$ are restrictions of invertible global sections of $\widehat{\mathcal{H}}_{r,g,n}^{\text{far}}$. This follows from the fact that $\mathcal{H}_{r,g,n}^{\text{far}}$ is the complement of the irreducible hypersurface $\mathrm{V}(\Phi_{r,g,rd+1})$ in $\widehat{\mathcal{H}}_{r,g,n}^{\text{far}}$. Moreover, we know that the projection to the first factor $q_1:\widehat{\mathcal{H}}_{r,g,n}^{\text{far}}\rightarrow\mathbb{A}^{n-3}\setminus\widetilde{\Delta}$ is flat, proper and with geometrically integral fibers, see Remark~\ref{rmk:integralgeometricfibers}. Therefore,
			\[
				{q_1}_*\mathcal{O}_{\widehat{\mathcal{H}}_{r,g,n}^{\text{far}}}\simeq\mathcal{O}_{\mathbb{A}^{n-3}\setminus\widetilde{\Delta}}
			\]
			hence all the invertible global sections of $\mathcal{H}_{r,g,n}^{\text{far}}$ come from $\mathbb{A}^{n-3}\setminus\widetilde{\Delta}$. It is immediate to see that these are of the form shown in the statement of this Lemma. Finally, if $\beta$ is set to 1, we have $2(n-3)+\binom{n-3}{2}=n(n-3)/2$ coefficients, which proves the statement.
		\end{proof}			
		\begin{proposition}\label{prop:injectivityfirstmap}
			Let $n\geq2$. Then, the sequence coming from~\eqref{eq:exactmain}
			\[
			\begin{tikzcd}
				0\arrow[r] & \mathbb{Z}^{(r-2)\binom{n}{2}+n}\arrow[r] & \mathrm{Pic}(\mathcal{H}_{r,g,n})\arrow[r] & \mathrm{Pic}(\mathcal{H}_{r,g,n}^{\text{far}})\arrow[r] & 0
			\end{tikzcd}
			\]
			is exact.
		\end{proposition}
		\begin{proof}
			Notice that if a relation between the $[\mathcal{Z}_{r,g,n}^{i,j,l}]$ holds for $n$ then it holds for all $n'\geq n$, of course changing the subscripts. Therefore, we may always assume $n\geq rd+1$.
			Let $\mathbb{Z}^{(r-1)\binom{n}{2}}\rightarrow\mathrm{Pic}(\mathcal{H}_{r,g,n})$ be the map induced by all the $[\mathcal{Z}_{r,g,n}^{i,j,l}]$. We already know that there exists a complex with injective first map
			\begin{equation}\label{eq:injectivityrelations}
				\begin{tikzcd}
					0\arrow[r] & \mathbb{Z}^{\binom{n}{2}-n}\arrow[r] & \mathbb{Z}^{(r-1)\binom{n}{2}}\arrow[r] & \mathrm{Pic}(\mathcal{H}_{r,g,n})
				\end{tikzcd}
			\end{equation}
			where $\mathbb{Z}^{\binom{n}{2}-n}$ is generated by the relations in Corollary~\ref{cor:relationsZgnij}. Moreover, the quotient of $\mathbb{Z}^{(r-1)\binom{n}{2}}$ by that subgroup is free of rank $(r-2)\binom{n}{2}+n$.
			We want to show that the above sequence is also exact in the middle. Let $K_{r,g,n}$ be the kernel of the second map. Since in a reduced and separated scheme two rational functions coinciding on an open dense subscheme are actually equal, we have an injective map
			\[
			\begin{tikzcd}
				K_{r,g,n}\arrow[r,hookrightarrow] & Q_{r,g,n}
			\end{tikzcd}
			\]
			sending a relation to the induced invertible global section in $\mathcal{H}_{r,g,n}^{\text{far}}$. By Lemma~\ref{lem:invertibleglobalsectionsHgnfarn>=2g+3}, we know that $Q_{r,g,n}\simeq\mathbb{Z}^{n(n-3)/2}$. On the other hand, $n(n-3)/2=\binom{n}{2}-n$, therefore the injective composition
			\[
			\begin{tikzcd}
				\mathbb{Z}^{\binom{n}{2}-n}\arrow[r,hookrightarrow] & K_{r,g,n}\arrow[r,hookrightarrow] & Q_{r,g,n}\simeq\mathbb{Z}^{n(n-3)/2}
			\end{tikzcd}
			\]
			is also surjective after tensoring by $\mathbb{Q}$, hence $K_{r,g,n}\simeq\mathbb{Z}^{\binom{n}{2}-n}$. Together with the fact that the quotient of $\mathbb{Z}^{(r-1)\binom{n}{2}}$ by the image of $\mathbb{Z}^{\binom{n}{2}-n}$ is free of rank $(r-2)\binom{n}{2}+n$, this implies that the sequence~\eqref{eq:injectivityrelations} is exact.
		\end{proof}
		\subsection{Reconstruction of $\mathrm{Pic}(\mathcal{H}_{r,g,n})$}
		Now, we try to construct a section of the homomorphism $\mathrm{Pic}(\mathcal{H}_{r,g,n})\rightarrow\mathrm{Pic}(\mathcal{H}_{r,g,n}^{\text{far}})$ for $n\geq3$, which will be possible only for $d$ even (for $n=2$ we already know that it exists, see Lemma~\ref{lem:comparisonpicardHg2farHg1}). To do this, we compare $\mathrm{Pic}(\mathcal{H}_{r,g,n}^{\text{far}})$ with $\mathrm{Pic}(\mathcal{H}_{r,g})$, computed in~\cite[Theorem 5.1]{AV04}. In the case $n=0$ the exact sequence~\eqref{eq:exactgeneralHgnfar} reads
		\[
		\begin{tikzcd}
			{[\Delta_{r,g,0}]}\mathbb{Z}\arrow[r] & \widehat{G}_{r,g,0}\arrow[r] & \mathrm{Pic}(\mathcal{H}_{r,g})\arrow[r] & 0
		\end{tikzcd}
		\]
		where $\widehat{G}_{r,g,0}\hookrightarrow\widehat{\mathrm{GL}_2}\simeq\widehat{\mathbb{G}_m}\simeq\mathbb{Z}$ generated by the determinant. The index of that subgroup is $d/2$ if $d$ is even, and $d$ if it is odd. See~\cite[Theorem 5.1]{AV04} for more details. Consider the map $q_{n,0}:\mathcal{H}_{r,g,n}^{\text{far}}\rightarrow\mathcal{H}_{r,g}$ given by forgetting the sections. Then, we have a commutative diagram
		\begin{equation}\label{diag:comparisonpicardHgnfarHg0partial}
			\begin{tikzcd}
				{[\Delta_{r,g,0}]}\mathbb{Z}\arrow[r] & \widehat{G}_{r,g,0}\arrow[r]\arrow[d] & \mathrm{Pic}(\mathcal{H}_{r,g})\arrow[r]\arrow[d,"q_{n,0}^*"] & 0\\
				{[\Delta_{r,g,n}]}\mathbb{Z}\arrow[r] & \widehat{G}_{r,g,n}\arrow[r] & \mathrm{Pic}(\mathcal{H}_{r,g,n}^{\mathrm{far}})\arrow[r] & 0
			\end{tikzcd}
		\end{equation}
		where $\widehat{G}_{r,g,0}\rightarrow\widehat{G}_{r,g,n}$ is induced by the inclusion $\mathbb{G}_m\subset\mathrm{GL}_2$ as the subgroup of diagonal matrices with equal non-zero coefficients. We know that the composite $\mathbb{Z}\simeq\widehat{\mathrm{GL}_2}\rightarrow\widehat{\mathbb{G}_m}\simeq\mathbb{Z}$ is the multiplication by 2. Therefore, $\widehat{G}_{r,g,0}\rightarrow\widehat{G}_{r,g,n}$ is an isomorphism if and only if $d$ is even, while it is just injective if $d$ is odd, with image of index 2. On the other hand, the image of $[\Delta_{r,g,0}]$ in $\widehat{G}_{r,g,n}\simeq\mathbb{Z}$ is always $2r(rd-1)$, hence we can complete the diagram~\eqref{diag:comparisonpicardHgnfarHg0partial} as
		\begin{equation}\label{diag:comparisonpicardHgnfarHg0complete}
			\begin{tikzcd}
				{[\Delta_{r,g,0}]}\mathbb{Z}\arrow[r]\arrow[d,"\simeq"] & \widehat{G}_{r,g,0}\arrow[r]\arrow[d] & \mathrm{Pic}(\mathcal{H}_{r,g})\arrow[r]\arrow[d,"q_{3,0}^*"] & 0\\
				{[\Delta_{r,g,n}]}\mathbb{Z}\arrow[r] & \widehat{G}_{r,g,n}\arrow[r] & \mathrm{Pic}(\mathcal{H}_{r,g,n}^{\mathrm{far}})\arrow[r] & 0
			\end{tikzcd}
		\end{equation}
		where the first vertical map is an isomorphism. Putting everything together, we get the following Lemma.
		\begin{lemma}\label{lemma:comparisonpicardHgnfarHg0}
			Let $3\leq n\leq rd+1$. Then, the composite
			\[
			\begin{tikzcd}[row sep=small, column sep=small]
				\mathcal{H}_{r,g,n}^{\mathrm{far}}\arrow[r,phantom,"\subset"] & \mathcal{H}_{r,g,n}\arrow[rr] & & \mathcal{H}_{r,g}
			\end{tikzcd}
			\]
			induces a commutative diagram
			\[
			\begin{tikzcd}[row sep=normal, column sep=small]
				\mathrm{Pic}(\mathcal{H}_{r,g})\arrow[rd]\arrow[rr,hookrightarrow] & & \mathrm{Pic}(\mathcal{H}^{\text{far}}_{r,g,n})\\
				& \mathrm{Pic}(\mathcal{H}_{r,g,n})\arrow[ur]
			\end{tikzcd}
			\]
			where the horizontal map is injective, and it is an isomorphism if $d$ is even.\qed
		\end{lemma}
		To extend the above Lemma to $n\geq rd+1$, we use Corollary~\ref{cor:injectivepicard}, which gives us the following commutative diagram
		\begin{equation}\label{diag:partial}
			\begin{tikzcd}
				& \mathrm{Pic}(\mathcal{H}^{\text{far}}_{r,g,n})\arrow[r,"\simeq"] &
				\mathrm{Pic}(\mathcal{H}^{\text{far}}_{r,g,n+1})\\
				\mathrm{Pic}(\mathcal{H}_{r,g})\arrow[ru,hookrightarrow]\arrow[rd]\\
				& \mathrm{Pic}(\mathcal{H}_{r,g,n})\arrow[uu,twoheadrightarrow]\arrow[r] & \mathrm{Pic}(\mathcal{H}_{r,g,n+1})\arrow[uu,twoheadrightarrow]
			\end{tikzcd}
		\end{equation}
		where the horizontal maps are induced by forgetting the last section. By induction, this immediately implies that Lemma~\ref{lemma:comparisonpicardHgnfarHg0} holds for all $n\geq rd+1$ too.
		\begin{lemma}\label{lem:strongcomparisonpicardHgnfarHg0}
			Let $n\geq3$. Then, the composite
			\[
			\begin{tikzcd}[row sep=small, column sep=small]
				\mathcal{H}_{r,g,n}^{\mathrm{far}}\arrow[r,phantom,"\subset"] & \mathcal{H}_{r,g,n}\arrow[rr] & & \mathcal{H}_{r,g}
			\end{tikzcd}
			\]
			induces a commutative diagram
			\[
			\begin{tikzcd}[row sep=normal, column sep=small]
				\mathrm{Pic}(\mathcal{H}_{r,g})\arrow[rd]\arrow[rr,hookrightarrow] & & \mathrm{Pic}(\mathcal{H}^{\text{far}}_{r,g,n})\\
				& \mathrm{Pic}(\mathcal{H}_{r,g,n})\arrow[ur]
			\end{tikzcd}
			\]
			where the horizontal map is injective, and it is an isomorphism if $d$ is even.\qed
		\end{lemma}		
		Therefore, in the case $d$ even the exact sequence of Proposition~\ref{prop:injectivityfirstmap} splits, and we obtain $\mathrm{Pic}(\mathcal{H}_{r,g,n})$.
		
		Now, suppose that $d$ is odd and consider the exact sequence~\eqref{eq:exactmain}
		\begin{equation*}
			\begin{tikzcd}
				0\arrow[r] & \mathbb{Z}^{(r-2)\binom{n}{2}+n}\arrow[r] & \mathrm{Pic}(\mathcal{H}_{r,g,n})\arrow[r] & \mathrm{Pic}(\mathcal{H}_{r,g,n}^{\text{far}})\arrow[r] & 0
			\end{tikzcd}
		\end{equation*}
		which we know to be exact on the left thanks to Proposition~\ref{prop:injectivityfirstmap}. Then, by Lemma~\ref{lem:strongcomparisonpicardHgnfarHg0},
		\[
			\mathrm{Pic}(\mathcal{H}_{r,g,n})\simeq\mathbb{Z}^{(r-2)\binom{n}{2}+n}\oplus T
		\]
		where $T$ is a torsion group injecting in $\mathrm{Pic}(\mathcal{H}_{r,g,n}^{\text{far}})$ and containing $\mathbb{Z}/r(rd-1)\mathbb{Z}$. Therefore, $T$ is a cyclic group of order $r(rd-1)$ or $2r(rd-1)$; we claim that only the first case occurs. To prove this, we have to introduce another natural divisor.
		\begin{construction}
			Let $n\geq1$. For all $1\leq i\leq n$, define $\mathcal{W}_{r,g,n}^i$ to be the divisor of $\mathcal{H}_{r,g,n}$ given by the points whose images under $\widetilde{\sigma}_i$ are in the ramification locus of the uniform cyclic covers parametrized by those points. Rigorously, it is defined as follows. Recall that $\mathcal{C}_{r,g}^n$ is the $n$-fold fibered product of $\mathcal{C}_{r,g}$ over $\mathcal{H}_{r,g}$. The ramification locus of the map $\mathcal{C}_{r,g}\rightarrow\mathcal{P}_{r,g}$ is a closed substack
			\[
				\mathcal{W}_{r,g,1}^1\subset\mathcal{C}_{r,g}=\mathcal{H}_{r,g,1}
			\]
			which is étale of degree $rd$ over $\mathcal{H}_{r,g}$, since $\mathcal{H}_{r,g}$ parametrizes smooth uniform cyclic covers of the projective line; see~\cite[Proposition 2.5]{AV04}. Then, for $n\geq2$, define
			\[
				\mathcal{W}_{r,g,n}^1:=(\mathcal{W}_{r,g,1}^1\times_{\mathcal{H}_{r,g}}\mathcal{C}_{r,g}^{n-1})\cap\mathcal{H}_{r,g,n}\subset\mathcal{H}_{r,g,n}\subset\mathcal{C}_{r,g}^n.
			\]
			For every other $i$ the definition of $\mathcal{W}_{r,g,n}^i$ is analogous; alternatively, one can consider the automorphism of $\mathcal{H}_{r,g,n}$ which exchanges the $i$-th section with the first.
		\end{construction}
		\begin{remark}
			Notice that for all $i$ the divisor $\mathcal{W}_{r,g,n}^i$ is smooth, since $\mathcal{W}_{r,g,1}^1$ is. In particular, $\mathcal{W}_{r,g,n}^i$ is geometrically reduced. Moreover, using our concrete presentations of $\mathcal{H}_{r,g,n}^{\text{far}}$, it is clear that for all $n$ and $i$ the divisor $\mathcal{W}_{r,g,n}^i\cap\mathcal{H}_{r,g,n}^{\text{far}}$ is defined by the equation $s_i=0$, hence it is irreducible (see also~\cite[Proposition 2.4]{EH21}). It follows that $\mathcal{W}_{r,g,n}^i$ is geometrically integral, since every point in $\mathcal{W}_{r,g,n}^i$ is part of a connected family in $\mathcal{W}_{r,g,n}^i$ intersecting $\mathcal{H}_{r,g,n}^{\text{far}}$. To see this, for $i=1$, take a point $(C\rightarrow\Spec K,p_1,\ldots,p_n)$ with $p_1$ a Weierstrass point and $K$ separably closed. Then we can consider the family $(X\rightarrow S,\sigma_1,\ldots,\sigma_n)$ defined as follows. Take $S$ to be the complement in $C^n$ of the extended diagonal and the points $(q_1,\ldots,q_n)$ with $q_j=p_1$ for some $j>1$, and let $X$ be its preimage in $C^n\times C$ under the projection. Define $\sigma_1$ to be the section with constant value $p_1$, and $\sigma_j$ the section induced by the projection $C^n\rightarrow C$ to the $j$-th component, for $j>1$. The fiber over $(p_1,\ldots,p_n)$ gives back our original point, and there exists a point of $S$ whose fiber corresponds to a point in $\mathcal{H}_{r,g,n}^{\text{far}}$. As $S$ is connected, we are done.
			
			See~\cite{EH21} and~\cite{EH22} for other properties of these stacks, for $r=2$.
			
			Finally, notice that the pullback of $\mathcal{W}_{r,g,n}^i$ under the morphism $\mathcal{H}_{r,g,n+1}\rightarrow\mathcal{H}_{r,g,n}$ which forgets $\widetilde{\sigma}_j$ is $\mathcal{W}_{r,g,n+1}^i$ if $j>i$, and $\mathcal{W}_{r,g,n+1}^{i+1}$ if $j\leq i$.
		\end{remark}
		\begin{lemma}\label{lem:W1n>=3}
			For all $n\geq3$, the image of \/ $\mathcal{W}_{r,g,n}^1$ in $\mathrm{Pic}(\mathcal{H}_{r,g,n}^{\text{far}})$ is a generator.
		\end{lemma}
		\begin{proof}
			Let $3\leq n\leq rd+1$. In the previous remark we have noticed that the restriction $\mathcal{W}_{r,g,n}^1\cap\mathcal{H}_{r,g,n}^{\text{far}}$ is defined in $(\mathbb{A}^{rd+1-n}\times(\mathbb{A}^{n-3}\setminus\widetilde{\Delta})\times\mathbb{A}^n)\setminus\Delta_{r,g,n}$ by the equation $s_1=0$. Since $\mathbb{G}_m/\mu_{d}$ acts with weight $-d$ on $s_1$, the result follows from Proposition~\ref{prop:picardHgnfar3<=n<=2g+3}.
			
			For $n\geq rd+1$ one uses Corollary~\ref{cor:injectivepicard} and the fact that the pullback of $\mathcal{W}_{r,g,n}^1$ under the morphism $\mathcal{H}_{r,g,n+1}\rightarrow\mathcal{H}_{r,g,n}$ forgetting the last section is $\mathcal{W}_{r,g,n+1}^1$.
		\end{proof}
		\begin{lemma}\label{lem:classW}
			Using the conventions of Theorem~\ref{thm:picardHg1}, if $d$ is even then $[\mathcal{W}_{r,g,1}^1]$ in $\mathrm{Pic}(\mathcal{H}_{r,g,1})$ corresponds to
			\[
			(-d/2,1)\in\mathbb{Z}\oplus\mathbb{Z}/2r(rd-1)\mathbb{Z}
			\]
			while if $d$ is odd then it corresponds to
			\[
			(-d,(d+1)/2)\in\mathbb{Z}\oplus\mathbb{Z}/r(rd-1)\mathbb{Z}.
			\]
			The same holds for $[\mathcal{W}_{r,g,2}^1]$ in $\mathrm{Pic}(\mathcal{H}_{r,g,2}^{\text{far}})$.
			
			Similarly, if $d$ is even then $[\mathcal{W}_{r,g,2}^2]$ in $\mathrm{Pic}(\mathcal{H}_{r,g,2}^{\text{far}})$ corresponds to
			\[
			(d/2,1)\in\mathbb{Z}\oplus\mathbb{Z}/2r(rd-1)\mathbb{Z}
			\]
			while if $d$ is odd then it corresponds to
			\[
			(d,-(d-1)/2)\in\mathbb{Z}\oplus\mathbb{Z}/r(rd-1)\mathbb{Z}.
			\]
		\end{lemma}	
		\begin{proof}
			Since $[\mathcal{W}_{r,g,1}^1]$ is defined by the equation $s_1=0$, in $\widehat{\mathrm{B}_2/\mu_{d}}$ it corresponds to $(0,d)$. Then, applying the transformations made in the proof of Theorem~\ref{thm:picardHg1}, we get the first part of the Lemma. Thanks to Lemma~\ref{lem:comparisonpicardHg2farHg1}, the statement holds also in $\mathrm{Pic}(\mathcal{H}_{r,g,2}^{\text{far}})$.
			
			For the second part, one uses again that $[\mathcal{W}_{r,g,2}^2]$ is defined by the equation $s_2=0$ in $\mathcal{H}_{r,g,n}^{\text{far}}$, hence in terms of the characters of $(\mathbb{G}_m\times\mathbb{G}_m)/\mu_{d}$ it corresponds to $(d,0)$, and one proceeds with the same calculations.
		\end{proof}
		The above Lemma shows that in $\mathrm{Pic}(\mathcal{H}_{g,2})$ the sum $[\mathcal{W}_{r,g,2}^1]+[\mathcal{W}_{r,g,2}^2]$ is expressible in terms of the $\mathcal{Z}_{r,g,2}^{1,2,l}$ and some torsion. Now, we compute the coefficients of the $\mathcal{Z}_{r,g,2}^{1,2,l}$ in such expression.
		We use the method of test curves, as done in~\cite{EH21}. To use it efficiently, we want to consider families of cyclic covers over a proper base, hence we cannot work on $\mathcal{H}_{r,g,2}$. There are two possible approaches, the first is to consider some compactification of $\mathcal{H}_{r,g,2}$ using stable curves, as done in~\cite{EH21}, the second uses the open immersion of $\mathcal{H}_{r,g,2}$ in $\mathcal{C}_{r,g}^2=\mathcal{C}_{r,g}\times_{\mathcal{H}_{r,g}}\mathcal{C}_{r,g}$ as the complement of the diagonal. We follow the second approach.
			
		First we compute the Picard group of $\mathcal{C}_{r,g}^2$. Notice that every divisor $\mathcal{W}_{r,g,n}^i$ and $\mathcal{Z}_{r,g,n}^{i,j,l}$ was defined as the intersection of some divisor in $\mathcal{C}_{r,g}^n$ with $\mathcal{H}_{r,g,n}$, so that we can see those divisors of $\mathcal{H}_{r,g,n}$ in $\mathcal{C}_{r,g}^{n}$ in a natural way. Denote by $\mathcal{Z}_{r,g,n}^{i,j,0}$ the diagonal of indices $i,j$ in $\mathcal{C}_{r,g}^n$.
		\begin{lemma}\label{lem:picCrg2}
			The open immersion of $\mathcal{H}_{r,g,2}$ in $\mathcal{C}_{r,g}^2$ induces an exact sequence
			\begin{equation}\label{eq:picCrg2}
				\begin{tikzcd}
					0\arrow[r] & {[\mathcal{Z}_{r,g,2}^{1,2,0}]}\mathbb{Z}\arrow[r] &\mathrm{Pic}(\mathcal{C}_{r,g}^2)\arrow[r] & \mathrm{Pic}(\mathcal{H}_{r,g,2})\arrow[r] & 0
				\end{tikzcd}
			\end{equation}
			which has a splitting.
		\end{lemma}
		\begin{proof}
			The exactness of the sequence follows from the Snake Lemma applied to the following commutative diagram with split exact rows
			\[
			\begin{tikzcd}
				0\arrow[r] & \bigoplus_{l=0}^{r-1}{[\mathcal{Z}_{r,g,2}^{1,2,l}]}\mathbb{Z}\arrow[d]\arrow[r] &\mathrm{Pic}(\mathcal{C}_{r,g}^2)\arrow[d,twoheadrightarrow]\arrow[r] & \mathrm{Pic}(\mathcal{H}_{r,g,2}^{\text{far}})\arrow[d,"="]\arrow[r] & 0\\
				0\arrow[r] & \bigoplus_{l=1}^{r-1}{[\mathcal{Z}_{r,g,2}^{1,2,l}]}\mathbb{Z}\arrow[r] &\mathrm{Pic}(\mathcal{H}_{r,g,2})\arrow[r] & \mathrm{Pic}(\mathcal{H}_{r,g,2}^{\text{far}})\arrow[r] & 0.
			\end{tikzcd}
			\]
			To show that the first row is exact one can proceed in the same way as what we have done for the bottom row. Indeed, using Lemma~\ref{lem:invertibleglobalsectionsHgnfarn>=2g+3} as in Proposition~\ref{prop:injectivityfirstmap} one can show that there are no relations in $\mathrm{Pic}(\mathcal{C}_{r,g}^n)$ between the $[\mathcal{Z}_{r,g,n}^{i,j,l}]$ except the ones we already know, for $n\geq rd+1$. Using the projection $\mathcal{C}_{r,g}^n\rightarrow\mathcal{C}_{r,g}^2$ to the first two factors, one gets the result for $n=2$ too.	Moreover, the upper row splits, as one can show considering an analogous diagram as in Lemma~\ref{lem:comparisonpicardHg2farHg1}.
			
			Finally, the splitting of~\eqref{eq:picCrg2} is induced by seeing $[\mathcal{Z}_{r,g,n}^{1,2,l}]$ as a divisor of $\mathcal{C}_{r,g}^2$, and by the splitting of the upper row of the above diagram.
		\end{proof}
		\begin{remark}
			One can argue in a similar way to obtain an analogous exact sequence for all $\mathcal{C}_{r,g}^n$, but we do not need it.
		\end{remark}
		\begin{lemma}\label{lem:relationWWdZ}
			In $\mathrm{Pic}(\mathcal{H}_{r,g,2})$,
			\[
				[\mathcal{W}_{r,g,2}^1]+[\mathcal{W}_{r,g,2}^2]-d\sum_{l=1}^{r-1}[\mathcal{Z}_{r,g,2}^{1,2,l}]
			\]
			is a torsion element.
		\end{lemma}
		\begin{proof}
			Consider the same problem in $\mathcal{C}_{r,g}^2$. Thanks to Lemma~\ref{lem:classW} and Lemma~\ref{lem:picCrg2}, we know that there exist $a_l\in\mathbb{Z}$ such that
			\[
				[\mathcal{W}_{r,g,2}^1]+[\mathcal{W}_{r,g,2}^2]-\sum_{l=0}^{r-1}a_l[\mathcal{Z}_{r,g,2}^{1,2,l}]
			\]
			is a torsion element in $\mathrm{Pic}(\mathcal{C}_{r,g}^2)$. Notice that both $[\mathcal{W}_{r,g,2}^1]$ and $[\mathcal{W}_{r,g,2}^2]$ are invariant under the action of $\mu_r$ on the second component of $\mathcal{C}_{r,g}^2$, while it permutes the classes $[\mathcal{Z}_{r,g,2}^{1,2,l}]$. Since there are no relations between these last classes, this shows that all $a_l$ are equal to some $a\in\mathbb{Z}$.
			
			Now, we show that $a=d$. We can work over an algebraically closed field $k$. Let $C\rightarrow\mathbb{P}_k^1$ be a smooth uniform cyclic cover of degree $r$ and genus $g$, and let $p\in C$ be a point where the map is unramified. Consider
			\[
				\begin{tikzcd}
					C\times C\arrow[r,"\mathrm{pr}_1"] & C 
				\end{tikzcd}
			\]
			with the two sections $\widetilde{\sigma}_1=\Delta$ and $\widetilde{\sigma}_2=1\times p$. This gives an object of $\mathcal{C}_{r,g}^2$ over $C$. Recall that $\alpha$ is a generator of $\mu_r$. Then $\mathcal{Z}_{r,g,2}^{1,2,l}$ in $C$ is given by $\alpha^{-l}(p)$, where $\alpha^0$ is just the identity. Moreover, $\mathcal{W}_{r,g,2}^2$ is 0, while the degree of $\mathcal{W}_{r,g,2}^1$ is $rd$. Equating the degrees, we get $rd=ra$, hence $a=d$, as wanted.
		\end{proof}
		\begin{remark}
			In the case $r=2$, Lemma~\ref{lem:relationWWdZ} is a corollary of~\cite[Theorem~1.1 and Theorem~1.2]{EH21}. In that case, the coefficient is $g+1$, which agrees with our result.
		\end{remark}
		\begin{lemma}\label{lem:torsiondodd}
			If $d$ is odd, then the torsion component of \/ $\mathrm{Pic}(\mathcal{H}_{r,g,n})$ is a cyclic group of order $r(rd-1)$.
		\end{lemma}
		\begin{proof}
			Consider the three pullbacks $\mathcal{H}_{r,g,3}\rightarrow\mathcal{H}_{r,g,2}$ given by forgetting one of the sections. Then the relation of Lemma~\ref{lem:relationWWdZ} gives us three similar relations, with different indices. Combining them, we get that, modulo torsion,
			\begin{align*}
				2[\mathcal{W}_{r,g,3}^1]&=([\mathcal{W}_{r,g,3}^1]+[\mathcal{W}_{r,g,3}^2])-([\mathcal{W}_{r,g,3}^2]+[\mathcal{W}_{r,g,3}^3])+([\mathcal{W}_{r,g,3}^1]+[\mathcal{W}_{r,g,3}^3])\\
				&=d\sum_{l=1}^{r-1}([\mathcal{Z}_{r,g,3}^{1,2,l}]-[\mathcal{Z}_{r,g,3}^{2,3,l}]+[\mathcal{Z}_{r,g,3}^{1,3,l}]).
			\end{align*}
			Recall that, for all $n\geq3$, in $\mathrm{Pic}(\mathcal{H}_{r,g,n})$ there are no other relations between the $\mathcal{Z}_{r,g,n}^{i,j,l}$ except those of Corollary~\ref{cor:relationsZgnij}. Since $d$ is not divisible by 2, it follows that $[\mathcal{W}_{r,g,n}^1]$ is not generated by the $\mathcal{Z}_{r,g,n}^{i,j,l}$ together with the torsion. It follows that the exact sequence of Proposition~\ref{prop:injectivityfirstmap} cannot be split in the case of $d$ odd, and we are done.
		\end{proof}
		Finally, putting everything together we get the main Theorem.		
		\begin{theorem}
			Suppose that the ground field $k$ is of characteristic not dividing $2rd$.
			\begin{itemize}
				\item If $d$ is even, then \/ $\mathrm{Pic}(\mathcal{H}_{r,g,n})\simeq\mathbb{Z}^{(r-2)\binom{n}{2}+n}\oplus\mathbb{Z}/2r(rd-1)\mathbb{Z}$.
				\item If $d$ is odd, then \/ $\mathrm{Pic}(\mathcal{H}_{r,g,n})\simeq\mathbb{Z}^{(r-2)\binom{n}{2}+n}\oplus\mathbb{Z}/r(rd-1)\mathbb{Z}$.
			\end{itemize}
			If $n<2$, we set $\binom{n}{2}$ to be 0. Moreover, the torsion part always comes from $\mathcal{H}_{r,g}$.\qed
		\end{theorem}
		\begin{proof}
			The case $n=0$ is already known, see~\cite[Theorem 5.1]{AV04}, while the case $n=1$ is exactly Theorem~\ref{thm:picardHg1}. The exact sequence~\eqref{eq:exactmain2}, Proposition~\ref{prop:injectivityfirstmap}, the splitting given by Lemma~\ref{lem:comparisonpicardHg2farHg1}, and Proposition~\ref{prop:picardHgnfar2} take care of the case $n=2$.
			The case $n\geq3$ reduces to reconstructing $\mathrm{Pic}(\mathcal{H}_{r,g,n})$ from the informations given by the exact sequence~\eqref{eq:exactmain} together with Proposition~\ref{prop:injectivityfirstmap}, Proposition~\ref{prop:picardHgnfar3<=n<=2g+3} and Proposition~\ref{prop:surjectivepicard}. We conclude using Lemma~\ref{lem:strongcomparisonpicardHgnfarHg0} and Lemma~\ref{lem:torsiondodd}, depending on the parity of $d$.
		\end{proof}
		It is worthwhile to explicitly state the result for the case of hyperelliptic curves, i.e. $r=2$, since it is the one of greatest geometric interest.
		\begin{corollary}
			Suppose that the ground field $k$ is of characteristic not dividing $2(g+1)$.
			\begin{itemize}
				\item If $g$ is odd, then \/ $\mathrm{Pic}(\mathcal{H}_{g,n})\simeq\mathbb{Z}^{n}\oplus\mathbb{Z}/(8g+4)\mathbb{Z}$.
				\item If $g$ is even, then \/ $\mathrm{Pic}(\mathcal{H}_{g,n})\simeq\mathbb{Z}^{n}\oplus\mathbb{Z}/(4g+2)\mathbb{Z}$.
			\end{itemize}
			In both cases, the torsion part comes from $\mathcal{H}_{g}$.\qed
		\end{corollary}
		\subsection{Generators of the Picard group}\label{subsec:generators}
		In this subsection we describe a minimal set of generators of the Picard group of $\mathcal{H}_{r,g,n}$ for $n\geq3$. The first description works only when $d$ is even.
		\begin{proposition}\label{prop:generatorsdeven}
			Let $n\geq3$ and let $d$ be even. Denote with $t$ a generator of $\mathrm{Pic}(\mathcal{H}_{r,g})$, which is hence of order $2r(rd-1)$. By abuse of notation we still denote its pullback in $\mathrm{Pic}(\mathcal{H}_{r,g,n})$ with the same letter, for all $n$. Then
			\[
			\{[\mathcal{Z}_{r,g,n}^{i,j,l}]\}_{\substack{1\leq i< j\leq n \\ 2\leq l\leq r-1}} \cup\{[\mathcal{Z}_{r,g,n}^{1,j,1}]\}_{2\leq j\leq n}\cup\{[\mathcal{Z}_{r,g,n}^{2,3,1}]\}\cup\{t\}
			\]
			is a minimal set of generators, with the only relation \/ $2r(rd-1)t=0$.
		\end{proposition}
		\begin{proof}
			It follows immediately from the results of Section~\ref{sec:finalcomputation}.
		\end{proof}		
		The case when $d$ is odd is different. Indeed, we know that the set given in the above proposition does not generate, since it misses $\mathcal{W}_{r,g,n}^1$; see the proof of Lemma~\ref{lem:torsiondodd}. Our results actually show that this is the only element missed, and that $t$ is expressible in terms of
		\[
		\{[\mathcal{Z}_{r,g,n}^{i,j,l}]\}_{\substack{1\leq i< j\leq n \\ 2\leq l\leq r-1}} \cup\{[\mathcal{Z}_{r,g,n}^{1,j,1}]\}_{2\leq j\leq n}\cup\{[\mathcal{Z}_{r,g,n}^{2,3,1}]\}\cup\{[\mathcal{W}_{r,g,n}^1]\}
		\]
		which is thus a set of generators of the Picard group. Again, there is only one relation (except for its multiples), which comes from Lemma~\ref{lem:relationWWdZ}. The same holds also in the $d$ even case, since $\mathcal{W}_{r,g,n}^1$ always gives a generator of $\mathrm{Pic}(\mathcal{H}_{r,g,n}^{\text{far}})$, see Lemma~\ref{lem:W1n>=3}.		
		\begin{proposition}\label{prop:generators}
			Let $n\geq3$. Then
			\[
			\{[\mathcal{Z}_{r,g,n}^{i,j,l}]\}_{\substack{1\leq i< j\leq n \\ 2\leq l\leq r-1}} \cup\{[\mathcal{Z}_{r,g,n}^{1,j,1}]\}_{2\leq j\leq n}\cup\{[\mathcal{Z}_{r,g,n}^{2,3,1}]\}\cup\{[\mathcal{W}_{r,g,n}^1]\}
			\]
			is a minimal set of generators. Moreover,
			\begin{equation}
				2r(rd-1)[\mathcal{W}_{r,g,n}^1]=rd(rd-1)\sum_{l=1}^{r-1}([\mathcal{Z}_{r,g,n}^{1,2,l}]+[\mathcal{Z}_{r,g,n}^{1,3,l}]-[\mathcal{Z}_{r,g,n}^{2,3,l}])
			\end{equation}
			is the only relation, regardless of the parity of $d$.\qed
		\end{proposition}

	\bibliographystyle{amsalpha}
	\bibliography{library}

\providecommand{\bysame}{\leavevmode\hbox to3em{\hrulefill}\thinspace}
\providecommand{\MR}{\relax\ifhmode\unskip\space\fi MR }
\providecommand{\MRhref}[2]{%
  \href{http://www.ams.org/mathscinet-getitem?mr=#1}{#2}
}
\providecommand{\href}[2]{#2}
\begin{thebibliography}{Mum77}

\bibitem[AC87]{AC87}
Enrico Arbarello and Maurizio Cornalba, \emph{The {P}icard groups of the moduli
  spaces of curves}, Topology \textbf{26} (1987), no.~2, 153--171. \MR{895568}

\bibitem[AV04]{AV04}
Alessandro Arsie and Angelo Vistoli, \emph{Stacks of cyclic covers of
  projective spaces}, Compos. Math. \textbf{140} (2004), no.~3, 647--666.
  \MR{2041774}

\bibitem[CL22]{CL22}
Samir Canning and Hannah Larson, \emph{The rational chow rings of moduli spaces
  of hyperelliptic curves with marked points}, Annali della Scuola Normale
  Superiore (2022).

\bibitem[EG98]{EG98}
Dan Edidin and William Graham, \emph{Equivariant intersection theory}, Invent.
  Math. \textbf{131} (1998), no.~3, 595--634. \MR{1614555}

\bibitem[EH21]{EH21}
Dan Edidin and Zhengning Hu, \emph{Chow classes of divisors on stacks of
  pointed hyperelliptic curves}, Annali della Scuola Normale Superiore (2021).

\bibitem[EH22]{EH22}
\bysame, \emph{The integral chow rings of the stacks of hyperelliptic
  weierstrass points}, arXiv preprint arXiv:2208.00556 (2022).

\bibitem[FO10]{FO10}
William Fulton and Martin Olsson, \emph{The {P}icard group of
  {$\mathcal{M}_{1,1}$}}, Algebra Number Theory \textbf{4} (2010), no.~1,
  87--104. \MR{2592014}

\bibitem[FV20]{FV20}
Roberto {Fringuelli} and Filippo {Viviani}, \emph{{On the Picard group scheme
  of the moduli stack of stable pointed curves}}, arXiv e-prints (2020),
  arXiv:2005.06920.

\bibitem[Gro65]{SGA}
Alexander Grothendieck, \emph{Cohomologie locale des faisceaux coh\'{e}rents et
  th\'{e}or\`emes de {L}efschetz locaux et globaux. {F}asc. {I}: {E}xpos\'{e}s
  1--8; {F}asc. {II}: {E}xpos\'{e}s 9--13}, Institut des Hautes \'{E}tudes
  Scientifiques, Paris, 1965, Troisi\`eme \'{e}dition, corrig\'{e}e,
  R\'{e}dig\'{e} par un groupe d'auditeurs, S\'{e}minaire de G\'{e}om\'{e}trie
  Alg\'{e}brique 1962. \MR{210718}

\bibitem[Har83]{Har83}
John Harer, \emph{The second homology group of the mapping class group of an
  orientable surface}, Invent. Math. \textbf{72} (1983), no.~2, 221--239.
  \MR{700769}

\bibitem[Mum65]{Mum65}
David Mumford, \emph{Picard groups of moduli problems}, Arithmetical
  {A}lgebraic {G}eometry ({P}roc. {C}onf. {P}urdue {U}niv., 1963), Harper \&
  Row, New York, 1965, pp.~33--81. \MR{201443}

\bibitem[Mum77]{Mum77}
\bysame, \emph{Stability of projective varieties}, Monographies de
  L'Enseignement Math\'{e}matique [Monographs of L'Enseignement
  Math\'{e}matique], vol. No. 24, L'Enseignement Math\'{e}matique, Geneva,
  1977, Lectures given at the ``Institut des Hautes \'{E}tudes Scientifiques'',
  Bures-sur-Yvette, March-April 1976. \MR{450273}

\bibitem[Per22]{Per22}
Michele Pernice, \emph{The integral {C}how ring of the stack of 1-pointed
  hyperelliptic curves}, Int. Math. Res. Not. IMRN (2022), no.~15,
  11539--11574. \MR{4458558}

\bibitem[PTT15]{PTT15}
Flavia Poma, Mattia Talpo, and Fabio Tonini, \emph{Stacks of uniform cyclic
  covers of curves and their {P}icard groups}, Algebr. Geom. \textbf{2} (2015),
  no.~1, 91--122. \MR{3322199}

\bibitem[Sca20]{Sca20}
Federico Scavia, \emph{Rational {P}icard group of moduli of pointed
  hyperelliptic curves}, Int. Math. Res. Not. IMRN (2020), no.~21, 8027--8056.
  \MR{4184614}

\bibitem[Vis98]{Vis98}
Angelo Vistoli, \emph{The {C}how ring of {$\mathcal{M}_2$}. {A}ppendix to
  ``{E}quivariant intersection theory'' [{I}nvent. {M}ath. {\bf 131} (1998),
  no. 3, 595--634; {MR}1614555 (99j:14003a)] by {D}. {E}didin and {W}.
  {G}raham}, Invent. Math. \textbf{131} (1998), no.~3, 635--644. \MR{1614559}

\end{thebibliography}
	
\end{document}